\documentclass{imanum}
\usepackage{graphicx}
\newcommand{\rmnum}[1]{\romannumeral #1}

\jno{drnxxx}

\begin{document}

\title{A new class of complex nonsymmetric algebraic Riccati equations}

\author{%
{\sc
Liqiang Dong\thanks{Email: dongliqiang@stu.xjtu.edu.cn},
Jicheng Li\thanks{Corresponding author. Email: jcli@mail.xjtu.edu.cn}
and
Xuenian Liu\thanks{Email: lxn901018@163.com}} \\[2pt]
School  of Mathematics and Statistics,
Xi'an Jiaotong University,\\  Xi'an, 710049, People's Republic of China
}

\shortauthorlist{Liqiang Dong \emph{et al.}}

\maketitle

\begin{abstract}
{In this paper, we first propose a new parameterized definition of comparison matrix of a given complex matrix, which generalizes the definition proposed by \cite {Axe1}. Based on this, we propose a new class of complex nonsymmetric algebraic Riccati equations (NAREs) which extends the class of nonsymmetric algebraic Riccati equations proposed by \cite {Axe1}. We also generalize the definition of the extremal solution of an NARE and show that the extremal solution of an NARE exists and is unique. Some classical algorithms can be applied to search for the extremal solution of an NARE, including Newton's method, some fixed-point iterative methods and doubling algorithms. Besides, we show that Newton's method is quadratically convergent and the fixed-point iterative method is linearly convergent. We also give some concrete strategies for choosing suitable parameters such that the doubling algorithms can be used to deliver the extremal solutions, and show that the two doubling algorithms with suitable parameters are quadratically convergent. Numerical experiments show that our strategies for parameters are effective.}
{parameterized comparison matrix; nonsymmetric algebraic Riccati equations (NAREs); the extremal solution; doubling algorithms; parameter selection strategy.}
\end{abstract}

\section{Introduction}
\label{sec;introduction}

A complex nonsymmetric algebraic Riccati equation (NARE) has the following form
\begin{eqnarray}\label{e1}
   XCX-XD-AX+B=0,
\end{eqnarray}
where $A\in \mathbb{C}^{m\times m}, B\in \mathbb{C}^{m\times n}, C\in \mathbb{C}^{n\times m}$ and $D\in \mathbb{C}^{n\times n}$ are known matrices and $X\in \mathbb{C}^{m\times n}$ is an unknown matrix. Let
\begin{eqnarray}\label{e2}
H=
\left(
\begin{matrix}
D&-C\\
B&-A
\end{matrix}
\right),~~~Q=JH=\left(
\begin{matrix}
D&-C\\
-B&A
\end{matrix}
\right),
\end{eqnarray}
where
$J=
\left(
\begin{matrix}
I_n&0\\
0&-I_m
\end{matrix}
\right)
$.
The complementary (dual) algebraic Riccati equation of the NARE (\ref{e1}) will be
\begin{eqnarray}\label{e3}
YBY-YA-DY+C=0,
\end{eqnarray}
which will be abbreviated as cARE, where $Y\in \mathbb{C}^{n\times m}$ is the unknown matrix.

In the following statements, some notations are necessary. $\textbf{1}$ denotes a column vector whose elements are all equal to one and size is implied by the context. $[A]_{ij}$ denotes the $(i,j)$ element of the matrix $A$.  The inequality $A\geq B$ ($A>B$) means that $[A]_{ij}\geq [B]_{ij}$ ($[A]_{ij}> [B]_{ij}$) for all $i,j$. For a square matrix $A$, we use $\rho(A)$ to denote its spectral radius and use ${\rm diag}(A)$ and ${\rm offdiag}(A)$ to denote its diagonal parts and off-diagonal parts, respectively. A real square matrix $A$ is called a nonsingular M-matrix if all of its off-diagonal elements are nonpositive and $Au>0$ for some positive vector $u$. Let ${\rm j}$ denote the imaginary unit. $|z|$, ${\rm Re}(z)$ and ${\rm Im}(z)$ denote the module, the real part and the imaginary part of the complex number $z$, respectively. For a complex matrix $A\in \mathbb{C}^{m\times n}$, $|A|$ is defined by $[|A|]_{ij}=|[A]_{ij}|$ for all $i,j$. Let $$z_\omega=-(1-\omega)+\omega {\rm j},~~~z_{\omega}^\bot=\omega+{\rm j}(1-\omega),$$
where $\omega$ is a real in the interval $[0,1]$. Then  $$L(z_\omega,\kappa)=lz_\omega+\frac{\kappa}{1-\omega}{\rm j},$$ where $\kappa$ is a given real number and $l$ takes all real numbers, represents a straight line passing through $\frac{\kappa}{1-\omega}{\rm j}$; and $R(z_{\omega}^\bot)= rz_{\omega}^\bot$, where $r$ takes all positive real numbers, represents a ray orthogonal to $L(z_\omega,\kappa)$. In particular, $L(z_\omega,0)=lz_\omega$, where $l$ takes all real numbers, represents a straight line passing through the origin. Obviously, $L(z_\omega,0)$ is parallel to $L(z_\omega,\kappa)$ for any nonzero real number $\kappa$.

In the literature, many researchers have studied the NARE (\ref{e1}). Some definitions  about comparison matrix of a given matrix are introduced as auxiliary tools. Definition \ref{thm39} below has been given by \cite {Axe1}. Definition \ref{thm31} below is usual definition of comparison matrix. Definition \ref{thm32} below is a generalization of Definition \ref{thm39}.

\begin{definition}{\rm (\cite {Axe1}).}\label{thm39}
The first comparison matrix of the complex square matrix $A$, will be denoted by $\widehat{A}$, is defined by
\begin{eqnarray}\label{e4}
[\widehat{A}]_{ij}=
\begin{cases}
{\rm Re}([A]_{ij}),&i=j,\\
-|[A]_{ij}|,&i\neq j.
\end{cases}
\end{eqnarray}
\end{definition}

\begin{definition}\label{thm31}
The second comparison matrix of the complex square matrix $A$, will be denoted by $\overline{A}$, is defined by
\begin{eqnarray}\label{e5}
[\overline{A}]_{ij}=
\begin{cases}
|[A]_{ij}|,&i=j,\\
-|[A]_{ij}|,&i\neq j.
\end{cases}
\end{eqnarray}
\end{definition}

\begin{definition}\label{thm32}
The $\omega$-comparison matrix of the complex square matrix $A$, will be denoted by $A_{\omega}$, is defined by
  \begin{eqnarray}\label{e6}
    [A_{\omega}]_{ij}=
    \left\{
    \begin{array}{lc}
    \omega{\rm Re}([A]_{ij})+(1-\omega){\rm Im}([A]_{ij}),&i=j,\\
    -|[A]_{ij}|,&i\neq j,
    \end{array}
    \right.
  \end{eqnarray}
where $\omega$ is a given real number in the interval $[0,1]$.
\end{definition}

Next, we give some simple terminologies which have been introduced in the literature. We call the NARE (\ref{e1}) is an M-matrix algebraic Riccati equation (MARE) if the matrix $Q$ defined by (\ref{e2}) is a nonsingular M-matrix or an irreducible singular M-matrix \cite[see.][]{Axe7}.  We call the NARE (\ref{e1}) is in class $\mathbb{M}$ if the matrix $Q$ is simply a nonsingular M-matrix \cite[see.][]{Axe3}. We call the NARE (\ref{e1}) is in class $\mathbb{H}^+$ if the second comparison matrix $\overline{Q}$ of the matrix $Q$ is a nonsingular M-matrix and the diagonal elements of the matrix $Q$ are positive. We call the NARE (\ref{e1}) is in class $\mathbb{H}^-$ if the second comparison matrix $\overline{Q}$ of the matrix $Q$ is a nonsingular M-matrix and the diagonal elements of the matrix $Q$ are negative. In fact, the NARE in class $\mathbb{H}^+$ and the NARE in class $\mathbb{H}^-$ can be transformed into each other \cite[see.][]{Axe2}. We call the NARE (\ref{e1}) is in class $\mathbb{H}^*$ if the first comparison matrix $\widehat{Q}$ of the matrix $Q$ is a nonsingular M-matrix.  We call the NARE (\ref{e1}) is in class $\mathbb{H}^\omega$ if the $\omega$-comparison matrix $Q_{\omega}$ of the matrix $Q$ is a nonsingular M-matrix. We call the NARE (\ref{e1}) is in class $\mathbb{H}$ if the second comparison matrix $\overline{Q}$ of the matrix $Q$ is a nonsingular M-matrix, i.e., $Q$ is an H-matrix. We observe that class $\mathbb{H}^\omega$ is a subclass of the class $\mathbb{H}$.

At present, many researchers have been studying NAREs. MAREs have been studied by many researchers, such as, \cite {Axe3} and so on. As is known to us all, an MARE has a unique minimal nonnegative solution and the solution can be obtained by many classical methods, mainly including the fixed-point iterative methods, Newton's method and doubling algorithms (SDA by \cite {Axe14}, SDA-ss by \cite {Axe12}, ADDA by \cite {Axe9}). The NAREs in class $\mathbb{H}^+$($\mathbb{H}^-$) have been studied by \cite {Axe2} . The NAREs in class $\mathbb{H}^*$ have been studied by \cite {Axe1} and \cite {Axe5}. The extremal solution of the NARE in class $\mathbb{H}^*$ exists and is unique and can also be obtained by the existing classical algorithms. We can see that the existing methods for the NAREs in class $\mathbb{M}$ can be applied to solve NAREs in class $\mathbb{H}^+$ ($\mathbb{H}^-$), in class $\mathbb{H}^*$. We can observe that Definition \ref{thm32} is reduced to Definition \ref{thm39} and $\mathbb{H}^\omega$ is reduced to $\mathbb{H}^*$ when $\omega=1$. The NAREs in class $\mathbb{H}^\omega$ for any $\omega \in [0,1]$ hasn't been studied except $\omega=1$. So, we will consider how to solve the NAREs in class $\mathbb{H}^\omega$ for general $\omega \in[0,1]$ in this paper.

The remainder of this paper is organized as follows. Section 2 gives some preliminaries which will be helpful for subsequent arguments. A generalized definition of the extremal solution of the NARE (\ref{e1}) in class $\mathbb{H}^\omega$ is given and the existence and uniqueness of the extremal solution of an NARE (\ref{e1}) in class $\mathbb{H}^\omega$ is proved in Section 3. In Section 4 and 5, Newton's method and fixed-point iterative methods are applied to deliver the extremal solution of the NARE (\ref{e1}) in class $\mathbb{H}^\omega$, respectively. In Section 6, two existing doubling algorithms, including ADDA and SDA, are applied to solve the NARE (\ref{e1}) in $\mathbb{H}^\omega$ by choosing suitable parameters. In Section 7, numerical experiments show that our algorithms are effective. Some concluding remarks are introduced to attract more valuable researches in Section 8.

\section{Preliminary knowledge}

\begin{lemma}\label{thm1}
Let $A\in \mathbb{R}^{n\times n}$ be a nonsingular M-matrix.

1. If $B\geq A$ is a $Z$-matrix, then $B$ is a nonsingular M-matrix with $B^{-1}\leq A^{-1}$.

2. Let $A$ be partitioned as
\begin{eqnarray*}
A=\left(
\begin{matrix}
A_{11}&A_{12}\\
A_{21}&A_{22}
\end{matrix}
\right),
\end{eqnarray*}
where $A_{11}$ and $A_{22}$ are square matrices, then $A_{11}$, $A_{22}$ and their Schur complements
$$A_{22}-A_{21}A_{11}^{-1}A_{12},~~~A_{11}-A_{12}A_{22}^{-1}A_{21},$$
are nonsingular M-matrices. In addition, if $A\textbf{1}>0$, then $A_{11}\textbf{1}>0$, $A_{22}\textbf{1}>0$, and
$$(A_{22}-A_{21}A_{11}^{-1}A_{12})\textbf{1}>0,~~~(A_{11}-A_{12}A_{22}^{-1}A_{21})\textbf{1}>0.$$
\end{lemma}

\begin{lemma}{\rm(\cite [see][]{Axe1})}.\label{thm2}
Let $A\in \mathbb{R}^{n\times n}$ be a nonsingular M-matrix. If $B\in \mathbb{C}^{n\times n}$ satisfies
that
  \begin{eqnarray*}
     |{\rm diag}(B)|\geq{\rm diag}(A),~~~|{\rm offdiag}(B)|\leq |{\rm offdiag}(A)|,
  \end{eqnarray*}
then $B$ is nonsingular with $|B^{-1}|\leq A^{-1}$.
\end{lemma}

\begin{lemma}\label{thm3}
Let $B\in \mathbb{C}^{n\times n}$ and suppose that $B_{\omega}$ is a nonsingular M-matrix. Then $B$ is nonsingular with
   \begin{eqnarray*}
     |B^{-1}|\leq (B_{\omega})^{-1}.
   \end{eqnarray*}
\end{lemma}

\begin{proof}
Because
\begin{eqnarray*}
\begin{aligned}
|{\rm diag}(B)|&=\sqrt{({\rm Re}({\rm diag}(B)))^2+({\rm Im}({\rm diag}(B)))^2}\\
&\geq\sqrt{\omega^2+(1-\omega)^2}\sqrt{({\rm Re}({\rm diag}(B)))^2+({\rm Im}({\rm diag}(B)))^2}\\
&\geq \omega {\rm Re}({\rm diag}(B))+(1-\omega){\rm Im}({\rm diag}(B))\\
&={\rm diag}(B_{\omega})
\end{aligned}
\end{eqnarray*}
and $$|{\rm offdiag}(B)|=|{\rm offdiag}(B_{\omega})|,$$
by Lemma \ref{thm2}, we have that $B$ is a nonsingular matrix with
\begin{eqnarray*}
 |B^{-1}|\leq (B_{\omega})^{-1}.
\end{eqnarray*}
\end{proof}

\begin{lemma}\label{thm4}
Let $B\in \mathbb{C}^{n\times n}$ and suppose that $B_{\omega}$ is a nonsingular M-matrix. Let $B$ be partitioned as
\begin{eqnarray*}
B=\left(
\begin{matrix}
B_{11}&B_{12}\\
B_{21}&B_{22}
\end{matrix}
\right),
\end{eqnarray*}
where $B_{11}$ and $B_{22}$ are square matrices. Let $S_{11}=B_{11}-B_{12}B_{22}^{-1}B_{21}$ and $S_{22}=B_{22}-B_{21}B_{11}^{-1}B_{12}$ be the Schur complement of $B_{22}$ and $B_{11}$, respectively. Then, $(B_{11})_{\omega}$, $(B_{22})_{\omega}$, $(S_{11})_{\omega}$ and $(S_{22})_{\omega}$ are nonsingular M-matrices. Moreover, if $B_{\omega} \textbf{1}>0$, then
$$(B_{11})_{\omega}\textbf{1}>0,~~(B_{22})_{\omega}\textbf{1}>0,~~(S_{11})_{\omega}\textbf{1}>0,~~(S_{22})_{\omega}\textbf{1}>0.$$
\end{lemma}

\begin{proof}
We have $$B_{\omega}=
\left(
\begin{matrix}
(B_{11})_{\omega}&-|B_{12}|\\
-|B_{21}|&(B_{22})_{\omega}
\end{matrix}
\right).$$
 By Lemma \ref{thm1}, we know that $(B_{11})_{\omega}$, $(B_{22})_{\omega}$ and their Schur complements,
$$(B_{22})_{\omega}-|B_{21}|((B_{11})_{\omega})^{-1}|B_{12}|,~~~(B_{11})_{\omega}-|B_{12}|((B_{22})_{\omega})^{-1}|B_{21}|$$
are also nonsingular M-matrices. By Lemma \ref{thm3}, $|B_{11}^{-1}|\leq ((B_{11})_{\omega})^{-1}, |B_{22}^{-1}|\leq ((B_{22})_{\omega})^{-1}$. Then,
\begin{eqnarray*}
\begin{aligned}
&~~~~{\rm diag}((S_{11})_{\omega})\\
&=\omega {\rm Re}({\rm diag}(S_{11}))+(1-\omega){\rm Im}({\rm diag}(S_{11}))\\
&=\omega {\rm Re}({\rm diag}(B_{11}))+(1-\omega){\rm Im}({\rm diag}(B_{11}))-\omega {\rm Re}({\rm diag}(B_{12}B_{22}^{-1}B_{21}))-(1-\omega){\rm Im}({\rm diag}(B_{12}B_{22}^{-1}B_{21}))\\
&\geq \omega {\rm Re}({\rm diag}(B_{11}))+(1-\omega){\rm Im}({\rm diag}(B_{11}))-|{\rm diag}(B_{12}B_{22}^{-1}B_{21})|\\
&\geq {\rm diag}((B_{11})_{\omega})-{\rm diag}(|B_{12}||((B_{22})_{\omega})^{-1}||B_{21}|)\\
&={\rm diag}((B_{11})_{\omega}-|B_{12}||((B_{22})_{\omega})^{-1}||B_{21}|),
\end{aligned}
\end{eqnarray*}
\begin{eqnarray*}
\begin{aligned}
{\rm offdiag}((S_{11})_{\omega})&=-|{\rm offdiag}((B_{11}-B_{12}B_{22}^{-1}B_{21})_{\omega})|\\
&\geq -|{\rm offdiag}(B_{11})|-|{\rm offdiag}(B_{12}B_{22}^{-1}B_{21})|\\
&\geq  {\rm offdiag}((B_{11})_{\omega})-{\rm offdiag}(|B_{12}||((B_{22})_{\omega})^{-1}||B_{21}|)\\
&={\rm offdiag}((B_{11})_{\omega}-|B_{12}||((B_{22})_{\omega})^{-1}||B_{21}|),
\end{aligned}
\end{eqnarray*}
which gives
$$(S_{11})_{\omega}\geq (B_{11})_{\omega}-|B_{12}||((B_{22})_{\omega})^{-1}||B_{21}|,$$
we have that $(S_{11})_{\omega}$ is a nonsingular M-matrix by Lemma \ref{thm1}. Similarly, $(S_{22})_{\omega}$ is also a nonsingular M-matrix.

If $B_{\omega}\textbf{1}>0$ , then $(B_{11})_{\omega}\textbf{1}>0$ and $((B_{11})_{\omega}-|B_{12}||((B_{22})_{\omega})^{-1}||B_{21}|)\textbf{1}>0$. Besides,
$$(S_{11})_{\omega}\textbf{1}\geq ((B_{11})_{\omega}-|B_{12}||((B_{22})_{\omega})^{-1}||B_{21}|)\textbf{1}>0.$$
Similarly, $(B_{22})_{\omega}\textbf{1}>0$ and $(S_{22})_{\omega}\textbf{1}>0$.

\end{proof}

\section{The existence and uniqueness of the extremal solution of the NARE (\ref{e1}) in class {\rm $\mathbb{H}^\omega$}}

\begin{theorem}\label{thm5}
Suppose the NARE {\rm(\ref{e1})} is in class {\rm $\mathbb{H}^\omega$} and $Q_{\omega}\textbf{1}>0$. Let $\widetilde{Q}$ be a nonsingular M-matrix satisfying $\widetilde{Q}\leq Q_{\omega}$ and $\widetilde{Q}\textbf{1}>0$, and partition $\widetilde{Q}$ as
\begin{eqnarray}\label{e8}
     \widetilde{Q}=
     \left(
     \begin{matrix}
     \widetilde{D}&-\widetilde{C}\\
     -\widetilde{B}&\widetilde{A}
     \end{matrix}
     \right),
   \end{eqnarray}
where $\widetilde{A}$, $\widetilde{B}$, $\widetilde{C}$ and $\widetilde{D}$ have same sizes as $A$, $B$, $C$ and $D$, respectively. Let $\widetilde{\Phi}$ and $\widetilde{\Psi}$ be the minimal nonnegative solutions of the NARE
   \begin{eqnarray}\label{e7}
     X\widetilde{C}X-X\widetilde{D}-\widetilde{A}X+\widetilde{B}=0
   \end{eqnarray}
and its cNARE
\begin{eqnarray}\label{e9}
     Y\widetilde{B}Y-Y\widetilde{A}-\widetilde{D}Y+\widetilde{C}=0,
   \end{eqnarray}
respectively. Then

1. The NARE {\rm(\ref{e1})} has a unique solution, denoted by $\Phi$, such that $|\Phi|\leq\widetilde{\Phi}$. Moreover, $D-C\Phi$ is a nonsingular matrix whose diagonal entries are in the upper right of $L(z_\omega,0)$, and $\Phi$ is the unique extremal solution of the NARE {\rm(\ref{e1})}, i.e., the solution such that the eigenvalues of $D-C\Phi$ are in the upper right of $L(z_\omega,0)$.

2. The NARE {\rm(\ref{e3})} has a unique solution, denoted by $\Psi$, such that $|\Psi|\leq\widetilde{\Psi}$. Moreover, $A-B\Psi$ is a nonsingular matrix whose diagonal entries are in the upper right of $L(z_\omega,0)$, and $\Psi$ is the unique extremal solution of the NARE {\rm(\ref{e3})}, i.e., the solution such that the eigenvalues of $A-B\Psi$ are in the upper right of $L(z_\omega,0)$.
\end{theorem}

\begin{proof}
We only prove Theorem  \ref{thm5} for the NARE (\ref{e1}). Similar arguments also work for the dual NARE (\ref{e3}) and thus omitted. Our proof is similar to that given by \cite {Axe1}. Here we only emphasize the differences.

The proof will be completed in four steps:

\rmnum{1}. Define the linear operators $\Upsilon$ and $\widetilde{\Upsilon}$ and prove that the invertibility of $\Upsilon$ and $\widetilde{\Upsilon}$.

\rmnum{2}. The mapping $f:\mathbb{C}^{m\times n}\rightarrow \mathbb{C}^{m\times n}$ given by
$f(Z)=\Upsilon^{-1}(ZCZ+B)$ has a fixed point in the set $\mathbb{S}=\{Z:|Z|\leq \widetilde{\Phi}\}$.

\rmnum{3}. $D-C\Phi$ is a nonsingular matrix whose diagonal entries are in the upper right of $L(z_\omega)$. Moreover, $\Phi$ is the extremal solution of the NARE (\ref{e1}), i.e., the solution such that the eigenvalues of $D-C\Phi$ are in the upper right of $L(z_\omega)$.

\rmnum{4}.  The uniqueness of the extremal solution.

(\rmnum{1}). Let the linear operators $\Upsilon,\widetilde{\Upsilon}: \mathbb{C}^{m\times n}\rightarrow \mathbb{C}^{m\times n}$ be defined by
   \begin{eqnarray}\label{e10}
     \Upsilon(X)=XD+AX,~~~~\widetilde{\Upsilon}(X)=X\widetilde{D}+\widetilde{A}X,
   \end{eqnarray}
respectively. The operator $\widetilde{\Upsilon}$ is invertible as $\widetilde{A}$ and $\widetilde{D}$ are nonsingular M-matrices. As
   $Q_{\omega}=
   \left(
   \begin{matrix}
   D_{\omega}&-|C|\\
   -|B|&A_{\omega}
   \end{matrix}
   \right)$
is a nonsingular M-matrix, $A_{\omega}$ and $D_{\omega}$ are also nonsingular M-matrices. We have
   \begin{eqnarray*}
     \widetilde{D}\leq D_{\omega},~\widetilde{A}\leq A_{\omega},~|C|\leq\widetilde{C},~|B|\leq\widetilde{B},
   \end{eqnarray*}
since $\widetilde{Q}\leq Q_{\omega}$.

By (\ref{e10}), we have ${\rm vec}(\Upsilon(X))={\rm vec}(XD+AX)=(D^T\otimes I+I\otimes A){\rm vec}(X)$. Because for $i=1,2,\cdots,n$; $j=1,2,\cdots,m$, we have
   \begin{eqnarray*}
     \begin{aligned}
    |[D]_{ii}+[A]_{jj}|&=\sqrt{({\rm Re}([D]_{ii})+{\rm Re}([A]_{jj}))^2+({\rm Im}([D]_{ii})+{\rm Im}([A]_{jj}))^2}\\
    &\geq \sqrt{\omega^2+(1-\omega)^2}\sqrt{({\rm Re}([D]_{ii})+{\rm Re}([A]_{jj}))^2+({\rm Im}([D]_{ii})+{\rm Im}([A]_{jj}))^2}\\
    &\geq \omega({\rm Re}([D]_{ii})+{\rm Re}([A]_{jj}))+(1-\omega)({\rm Im}([D]_{ii})+{\rm Im}([A]_{jj}))\\
    &=(\omega{\rm Re}([D]_{ii})+(1-\omega){\rm Im}([D]_{ii}))+(\omega{\rm Re}([A]_{jj})+(1-\omega){\rm Im}([A]_{jj})).
     \end{aligned}
   \end{eqnarray*}
This leads to
   \begin{eqnarray*}
     \begin{aligned}
     |{\rm diag}(D^T\otimes I+I\otimes A)|&\geq{\rm diag}((D_{\omega})^T\otimes I +I\otimes A_{\omega})\geq {\rm diag}(\widetilde{D}^T\otimes I +I\otimes \widetilde{A}),\\
     |{\rm offdiag}(D^T\otimes I +I\otimes A)|&\leq|{\rm offdiag}((D_{\omega})^T\otimes I +I\otimes A_{\omega})|\leq |{\rm offdiag}(\widetilde{D}^T\otimes I +I\otimes \widetilde{A})|.
     \end{aligned}
   \end{eqnarray*}
By Lemma \ref{thm2}, we have that $D^T\otimes I +I\otimes A$ is nonsingular with
$|(D^T\otimes I +I\otimes A)^{-1}|\leq(\widetilde{D}^T\otimes I +I\otimes \widetilde{A})^{-1}$. So, the operator $\Upsilon$ is also invertible.

(\rmnum{2}). The proof is similar to that of Theorem 3.1 given  by \cite {Axe1}, thus omitted.

(\rmnum{3}). Let $R=D-C\Phi$ and $\widetilde{R}=\widetilde{D}-\widetilde{C}\widetilde{\Phi}$. As $\widetilde{Q}\textbf{1}>0$, $\widetilde{R}$ is a nonsingular M-matrix with $\widetilde{R}\textbf{1}>0$.
For $i=1,2,\cdots,n$, we have $[\widetilde{R}]_{ii}=[\widetilde{D}]_{ii}-[\widetilde{C}\widetilde{\Phi}]_{ii}>\sum_{j=1,j\neq i}^{n}([\widetilde{C}\widetilde{\Phi}]_{ij}-[\widetilde{D}]_{ij})\geq0$
and
   $\sum_{j=1,j\neq i}^{n}|[R]_{ij}|<[\widetilde{D}]_{ii}-[\widetilde{C}\widetilde{\Phi}]_{ii}.$
On the other hand, we have
\begin{eqnarray}\label{e11}
\begin{aligned}
&~~~~\omega {\rm Re}([R]_{ii})+(1-\omega){\rm Im}([R]_{ii})\\
&=\omega {\rm Re}([D]_{ii})+(1-\omega){\rm Im}([D]_{ii})-\omega {\rm Re}([C\Phi]_{ii})-(1-\omega){\rm Im}([C\Phi]_{ii})\\
&\geq [\widetilde{D}]_{ii}-|[C\Phi]_{ii}|\\
&\geq[\widetilde{D}]_{ii}-[\widetilde{C}\widetilde{\Phi}]_{ii}\\
&>\sum_{j=1,j\neq i}^{n}|[R]_{ij}|,~~i=1,2,\cdots,n.
\end{aligned}
\end{eqnarray}
By (\ref{e11}), we can say that $[R]_{ii}$ are in the upper right of $L(z_\omega,\sum_{j=1,j\neq i}^{n}|[R]_{ij}|),i=1,2,\cdots,n$. The distance between $L(z_\omega,0)$ and $L(z_\omega,\sum_{j=1,j\neq i}^{n}|[R]_{ij}|)$ can be expressed as $\frac{\sum_{j=1,j\neq i}^{n}|[R]_{ij}|}{\sqrt{\omega^2+(1-\omega)^2}}\geq \sum_{j=1,j\neq i}^{n}|[R]_{ij}|$ since $\omega^2+(1-\omega)^2\leq1$ for any $\omega\in[0,1]$. That is to say, the distance between $[R]_{ii}$ and $L(z_\omega,0)$ is greater or equal to $\sum_{j=1,j\neq i}^{n}|[R]_{ij}|$ for each $i$. So it follows from Gershgorin's theorem that the eigenvalues of $R$ are in the upper right of $L(z_\omega,0)$.

(\rmnum{4}). Let $S=A-B\Psi$. Similar to the proof of (\rmnum{3}), it is easy to show that the eigenvalues of $S$ are in the upper right of $L(z_\omega,0)$.
Because
   \begin{eqnarray*}
   H
   \left(
   \begin{matrix}
   I\\
   \Phi
   \end{matrix}
   \right)
   =
   \left(
   \begin{matrix}
   I\\
   \Phi
   \end{matrix}
   \right)R,~~~
   H
   \left(
   \begin{matrix}
   \Psi\\
   I
   \end{matrix}
   \right)
   =\left(
   \begin{matrix}
   \Psi\\
   I
   \end{matrix}
   \right)(-S),
   \end{eqnarray*}
$\Phi$ is determined by an invariant subspace of $H$ corresponding to $n$ eigenvalues in the upper right of $L(z_\omega,0)$, and $\Psi$ is determined by an invariant subspace of $H$ corresponding to $m$ eigenvalues in the lower left of $L(z_\omega,0)$.  Note that $Q_{\omega}\textbf{1}>0$, it follows from Gershgorin's theorem that
$$\left(
\begin{matrix}
D_{\omega}&-|C|\\
|B|&-A_{\omega}
\end{matrix}
\right)$$
has exactly $n$ eigenvalues in $\mathbb{C}^+$ and $m$ eigenvalues in $\mathbb{C}^-$. Similarly to proof of (\rmnum{3}), $H$ has exactly $n$ eigenvalues in the upper right of $L(z_\omega,0)$ and $m$ eigenvalues in the lower left of $L(z_\omega,0)$.  Consequently, $\Phi$ is the unique solution such that the eigenvalues of $D-C\Phi$ are in the upper right of $L(z_\omega,0)$ and $\Psi$ is the unique solution such that the eigenvalues of $A-B\Psi$ are in the upper right of $L(z_\omega,0)$.
\end{proof}

\begin{remark}\label{thm9}

1. When $\omega=0$, $z_\omega=-1$, the straight line $L(z_\omega,0)=-l$, where $l$ takes all real numbers. Obviously, the line will be the real axis. The upper right of $L(z_\omega,0)$ will be upper half plane of the whole complex plane.

2. When $\omega=1$, $z_\omega={\rm j}$, the straight line $L(z_\omega,0)={\rm j}l$, where $l$ takes all real numbers. Obviously, the line will be the imaginary axis. The upper right of $L(z_\omega,0)$ will be $\mathbb{C}^+$.

3. When $\omega$ ranges from $0$ to $1$, the straight line $L(z_\omega,0)$ rotates clockwise from the real axis to the imaginary axis at the center of origin, and the the upper right of $L(z_\omega,0)$  rotates clockwise from the upper half plane of the whole complex plane to the half plane $\mathbb{C}^+$  at the center of origin. So we can say that our results generalize the results of Theorem 3.1 given by \cite {Axe1}.

\end{remark}

\section{Apply Newton's method to solve the NARE (\ref{e1}) in class $\mathbb{H}^{\omega}$}
Applying Newton's method to solve the NARE (\ref{e1}) in class $\mathbb{H}^\omega$, we get the iterative scheme:
   \begin{eqnarray}\label{e12}
   \begin{aligned}
     \Phi_0=0,~~(A-\Phi_kC)\Phi_{k+1}+\Phi_{k+1}(D-C\Phi_{k})=B-\Phi_{k}C\Phi_{k},~k\geq0.
     \end{aligned}
   \end{eqnarray}
In this section we will show that $\{\Phi_k\}$ quadratically converges to  the extremal solution $\Phi$ of the NARE (\ref{e1}) in class $\mathbb{H}^\omega$.

To achieve this, we consider Newton's iteration for the NARE (\ref{e7}) in class $\mathbb{M}$:
   \begin{eqnarray}\label{e13}
   \begin{aligned}
     \widetilde{\Phi}_0=0,~~(\widetilde{A}-\widetilde{\Phi}_k\widetilde{C})\widetilde{\Phi}_{k+1}+\widetilde{\Phi}_{k+1}(\widetilde{D}-\widetilde{C}\widetilde{\Phi}_{k})&=\widetilde{B}-\widetilde{\Phi}_{k}\widetilde{C}\widetilde{\Phi}_{k},~k\geq0.
     \end{aligned}
   \end{eqnarray}
It has been shown by \cite {Axe14} that $\{\widetilde{\Phi}_k\}$ generated by (\ref{e13}) increasingly quadratically converges to the minimal nonnegative solution $\widetilde{\Phi}$ of the NARE (\ref{e7}). We first compare the increments at each iterative step of the two iterations (\ref{e12}) and (\ref{e13}) as in Lemma \ref{thm6} below. Its proof is slightly different from that of Lemma 4.1 given by \cite {Axe1}. Here we give its detailed proof for readers' convenience.

\begin{lemma}\label{thm6}
Let the sequences $\{\Phi_k\}$ and $\{\widetilde{\Phi}_k\}$ be generated by the iterations {\rm(\ref{e12})}  for the NARE {\rm(\ref{e1})} in class $\mathbb{H}^\omega$ and {\rm(\ref{e13})}  for the NARE {\rm(\ref{e7})} in class $\mathbb{M}$, respectively. Let $Q$ and $\widetilde{Q}$ be as in {\rm(\ref{e2})} and {\rm(\ref{e8})}, respectively. Suppose that $Q_\omega$ and $\widetilde{Q}$ are nonsingular M-matrices satisfying
   $$Q_\omega\textbf{1}>0,~~\widetilde{Q}\textbf{1}>0,~~\widetilde{Q}\leq Q_\omega.$$
Then we have
   $$|\Phi_{k+1}-\Phi_k|\leq \widetilde{\Phi}_{k+1}-\widetilde{\Phi}_k,~k=1,2,\cdots.$$
\end{lemma}

\begin{proof}
Notice that $\widetilde{\Phi}_k\leq \widetilde{\Phi}_{k+1}$ and  $I\otimes (\widetilde{A}-\widetilde{\Phi}_k\widetilde{C})+(\widetilde{D}-\widetilde{C}\widetilde{\Phi}_k)^T\otimes I$ is a nonsingular M-matrix for each nonnegative integer $k$. The proof goes on by mathematical induction on $k$. For $k=0$, we have
$$(I\otimes A+D^T\otimes I) {\rm vec}(\Phi_1)={\rm vec}(B),~~~(I\otimes \widetilde{A}+\widetilde{D}^T\otimes I) {\rm vec}(\widetilde{\Phi}_1)={\rm vec}(\widetilde{B}).$$
By the proof of Theorem  \ref{thm5}, we can obtain $$|(I\otimes A+D^T\otimes I)^{-1}|\leq (I\otimes \widetilde{A}+\widetilde{D}^T\otimes I)^{-1},$$
which results in $|\Phi_1|\leq \widetilde{\Phi}_1$. For $k\geq 1$, we have
   \begin{eqnarray*}
   \begin{aligned}
   &(A-\Phi_kC)(\Phi_{k+1}-\Phi_k)+(\Phi_{k+1}-\Phi_k)(D-C\Phi_k)=(\Phi_{k}-\Phi_{k-1})C(\Phi_{k}-\Phi_{k-1}),\\
   &(\widetilde{A}-\widetilde{\Phi}_k\widetilde{C})(\widetilde{\Phi}_{k+1}-\widetilde{\Phi}_k)
   +(\widetilde{\Phi}_{k+1}-\widetilde{\Phi}_k)(\widetilde{D}-\widetilde{C}\widetilde{\Phi}_k)
   =(\widetilde{\Phi}_{k}-\widetilde{\Phi}_{k-1})\widetilde{C}(\widetilde{\Phi}_{k}-\widetilde{\Phi}_{k-1}).
   \end{aligned}
   \end{eqnarray*}
Suppose that $|\Phi_{k+1}-\Phi_{k}|\leq\widetilde{\Phi}_{k+1}-\widetilde{\Phi}_{k}$ for $k\leq l-1$. Then,
   $$|\Phi_k|\leq \sum_{j=1}^{k}|\Phi_{k}-\Phi_{k-1}|\leq
   \sum_{j=1}^{k}(\widetilde{\Phi}_{k}-\widetilde{\Phi}_{k-1})=\widetilde{\Phi}_k,~1\leq k\leq l.$$
Because $|{\rm diag}(I\otimes A+D^T\otimes I)|\geq|{\rm diag}(I\otimes A_{\omega}+(D_\omega)^T\otimes I)|\geq {\rm diag}(I\otimes \widetilde{A}+\widetilde{D}^T\otimes I)$, we have
   \begin{eqnarray*}
   \begin{aligned}
   &~~~~|{\rm diag}(I\otimes(A-\Phi_lC)+(D-C\Phi_l)^T\otimes I)|\\
   &\geq|{\rm diag}(I\otimes A+D^T\otimes I)|-|{\rm diag}(I\otimes\Phi_lC+(C\Phi_l)^T\otimes I)|\\
   &\geq{\rm diag}(I\otimes \widetilde{A}+\widetilde{D}^T\otimes I)-{\rm diag}(I\otimes\widetilde{\Phi}_l\widetilde{C}+(\widetilde{C}\widetilde{\Phi}_l)^T\otimes I)\\
   &={\rm diag}(I\otimes (\widetilde{A}-\widetilde{\Phi}_l\widetilde{C})+(\widetilde{D}-\widetilde{C}\widetilde{\Phi}_l)^T\otimes I)
   \end{aligned}
   \end{eqnarray*}
and
   \begin{eqnarray*}
   \begin{aligned}
   &~~~~|{\rm offdiag}(I\otimes (A-\Phi_lC)+(D-C\Phi_l)^T\otimes I)|\\
   &\leq|{\rm offdiag}(I\otimes A)|+|{\rm offdiag}(I\otimes(\Phi_lC))|+|{\rm offdiag}(D^T\otimes I)|+|{\rm offdiag}((C\Phi_l)^T\otimes I)|\\
   &\leq-{\rm offdiag}(I\otimes \widetilde{A})+|{\rm offdiag}(I\otimes(\widetilde{\Phi}_l\widetilde{C}))|-{\rm offdiag}(\widetilde{D}^T\otimes I)+|{\rm offdiag}((\widetilde{C}\widetilde{\Phi}_l)^T\otimes I)|\\
   &=|{\rm offdiag}(I\otimes(\widetilde{A}-\widetilde{\Phi}_l\widetilde{C}))|+|{\rm offdiag}((\widetilde{D}-\widetilde{C}\widetilde{\Phi}_l)^T\otimes I)|\\
   &=|{\rm offdiag}(I\otimes (\widetilde{A}-\widetilde{\Phi}_l\widetilde{C})+(\widetilde{D}-\widetilde{C}\widetilde{\Phi}_l)^T\otimes I)|.
   \end{aligned}
   \end{eqnarray*}
Thus we have
   $$|(I\otimes(A-\Phi_lC)+(D-C\Phi_l)^T\otimes I)^{-1}|\leq (I\otimes (\widetilde{A}-\widetilde{\Phi}_l\widetilde{C})+(\widetilde{D}-\widetilde{C}\widetilde{\Phi}_l)^T\otimes I)^{-1}$$
and
   $$(\Phi_{l+1}-\Phi_l)=(I\otimes(A-\Phi_lC)+(D-C\Phi_l)^T\otimes I)^{-1}(\Phi_{l}-\Phi_{l-1})C(\Phi_{l}-\Phi_{l-1}).$$
By the induction hypothesis,
   $$|(\Phi_l-\Phi_{l-1})C(\Phi_l-\Phi_{l-1})|\leq (\widetilde{\Phi}_l-\widetilde{\Phi}_{l-1})\widetilde{C}(\widetilde{\Phi}_l-\widetilde{\Phi}_{l-1}).$$
Therefore, we have $$|\Phi_{l+1}-\Phi_{l}|\leq \widetilde{\Phi}_{l+1}-\widetilde{\Phi}_l.$$

\end{proof}

After establishing Lemma \ref{thm6}, we can easily prove that the sequence $\{\Phi_k\}$ quadratically converges to the extremal solution $\Phi$ of the NARE (\ref{e1}) in class $\mathbb{H}^\omega$ as in Theorem \ref{thm7} below. Although the proof of Theorem \ref{thm7} is similar to that of Theorem 4.1 given by \cite {Axe1}, thus omitted, the connotation of Theorem \ref{thm7} is more abundant than that of Theorem 4.1 given by \cite {Axe1}.

\begin{theorem}\label{thm7}
The sequence $\{\Phi_k\}$ generated by the iterations {\rm(\ref{e12})} for the NARE {\rm(\ref{e1})} in class $\mathbb{H}^\omega$ quadratically converges to the extremal solution $\Phi$.
\end{theorem}

\section{Apply the fixed-point iterative methods to solve the NARE (\ref{e1}) in class $\mathbb{H}^\omega$}
Based on the two splittings
   $$A=A_1-A_2,~D=D_1-D_2,$$
we derive the fixed-point iterative scheme
   \begin{eqnarray}\label{e14}
   \begin{aligned}
     \Phi_0=0,~~A_1\Phi_{k+1}+\Phi_{k+1}D_1=B+A_2\Phi_{k}+\Phi_{k}D_2+\Phi_{k}C\Phi_{k},~k\geq0,
     \end{aligned}
   \end{eqnarray}
for the NARE (\ref{e1}) in class $\mathbb{H}^\omega$. Theorem \ref{thm8} below gives a sufficient condition under which  the sequence $\{\Phi_k\}$ generated by the fixed-point iteration (\ref{e14}) is linearly convergent to the extremal solution $\Phi$ of the NARE (\ref{e1}) in class $\mathbb{H}^\omega$.

\begin{theorem}\label{thm8}
Suppose that $Q_{\omega}\textbf{1}>0$, let
   $$\widetilde{A}=(A_1)_{\omega}-|A_2|,~~\widetilde{D}=(D_1)_{\omega}-|D_2|,~~\widetilde{B}=|B|,~~\widetilde{C}=|C|.$$
If
   $$I\otimes (A_1)_{\omega}+((D_1)_{\omega})^T\otimes I$$
is a nonsingular M-matrix, and $\widetilde{Q}$ is a nonsingular M-matrix satisfying
   $$\widetilde{Q}\leq Q_{\omega},~~~\widetilde{Q}\textbf{1}>0,$$
then  the sequence $\{\Phi_k\}$ generated by the iterative scheme {\rm(\ref{e14})}
linearly converges to the extremal solution $\Phi$ of the NARE {\rm(\ref{e1})} in class $\mathbb{H}^{\omega}$.
\end{theorem}

\begin{proof}
Consider the fixed-point iteration
   \begin{eqnarray}\label{e15}
   \begin{aligned}
     \widetilde{\Phi}_0=0,~~(A_1)_{\omega}\widetilde{\Phi}_{k+1}+\widetilde{\Phi}_{k+1}(D_1)_{\omega}=\widetilde{B}+|A_2|\widetilde{\Phi}_k+\widetilde{\Phi}_k|D_2|+\widetilde{\Phi}_k\widetilde{C}\widetilde{\Phi}_k,~k\geq0,
   \end{aligned}
   \end{eqnarray}
applied to the NARE (\ref{e7}) in class $\mathbb{M}$. It is well-known that $\widetilde{\Phi}_k\leq \widetilde{\Phi}_{k+1}$ for each nonnegative integer $k$ and ${\rm lim}_{k\rightarrow\infty}\widetilde{\Phi}_k=\widetilde{\Phi}$,
where $\widetilde{\Phi}$ is the minimal nonnegative solution of the NARE (\ref{e7}) in class $\mathbb{M}$  \citep [see.] []{Axe13}. For $k=0$, we have
$$A_1\Phi_1+\Phi_1D_1=B,~~~~(A_1)_{\omega}\widetilde{\Phi}_1+\widetilde{\Phi}_1 (D_1)_{\omega}=\widetilde{B}.$$
We can prove that
   $$|(I\otimes A_1+D_1^T\otimes I)^{-1}|\leq (I\otimes (A_1)_{\omega}+((D_1)_{\omega})^T\otimes I)^{-1}$$
by verifying
   \begin{eqnarray*}
   \begin{aligned}
   |{\rm diag}(I\otimes A_1+D_1^T\otimes I)|&\geq {\rm diag}(I\otimes (A_1)_{\omega}+((D_1)_{\omega})^T\otimes I),\\
   |{\rm offdiag}(I\otimes A_1+D_1^T\otimes I)|&\leq |{\rm offdiag}(I\otimes (A_1)_{\omega}+((D_1)_{\omega})^T\otimes I)|,
   \end{aligned}
   \end{eqnarray*}
which results in $|\Phi_1|\leq\widetilde{\Phi}_1$. For $k\geq1$, it follows from the two iterations (\ref{e14}) and (\ref{e15}) that
   \begin{eqnarray*}
   \begin{aligned}
   &~~~~A_1(\Phi_{k+1}-\Phi_k)+(\Phi_{k+1}-\Phi_k)D_1\\
   &=A_2(\Phi_{k}-\Phi_{k-1})+(\Phi_{k}-\Phi_{k-1})D_2+\Phi_kC(\Phi_{k}-\Phi_{k-1})+(\Phi_{k}-\Phi_{k-1})C\Phi_{k-1},\\
   &~~~~(A_1)_{\omega}(\widetilde{\Phi}_{k+1}-\widetilde{\Phi}_k)+(\widetilde{\Phi}_{k+1}-\widetilde{\Phi}_k)((D_1)_{\omega})^T\\
   &=|A_2|(\widetilde{\Phi}_{k}-\widetilde{\Phi}_{k-1})
   +(\widetilde{\Phi}_{k}-\widetilde{\Phi}_{k-1})|D_2|+\widetilde{\Phi}_k\widetilde{C}(\widetilde{\Phi}_{k}-\widetilde{\Phi}_{k-1})+(\widetilde{\Phi}_{k}-\widetilde{\Phi}_{k-1})\widetilde{C}\widetilde{\Phi}_{k-1}.
   \end{aligned}
   \end{eqnarray*}
Using the similar way that we prove Lemma \ref{thm6}, we can show inductively that
   $$|\Phi_{k+1}-\Phi_k|\leq\widetilde{ \Phi}_{k+1}-\widetilde{\Phi}_k,~~k\geq 0,$$
and
   $${\rm lim}_{k\rightarrow \infty }\Phi_k=\Phi.$$

Next, we consider the convergence speed of the sequence $\{\Phi_k\}$. Similar to the proof of Theorem 3.2 given by \cite{Axe13}, we can prove that, for the fixed-point iteration (\ref{e14}) with $\Phi_0=0$,
   $${\rm lim}_{k\rightarrow \infty}{\rm sup}\sqrt[k]{\parallel \Phi_k-\Phi\parallel}=\rho((I\otimes A_1+D_1^T\otimes I)^{-1}(I\otimes(A_2+\Phi C)+(D_2+C\Phi)^T\otimes I)).$$
Because
   $$M_{\widetilde{\Phi}}=(I\otimes (A_1)_{\omega}+((D_1)_{\omega})^T\otimes I)-(I\otimes(|A_2|+\widetilde{\Phi} \widetilde{C})+(|D_2|+\widetilde{C}\widetilde{\Phi})^T\otimes I)$$
is a regular splitting of a nonsingular M-matrix $M_{\widetilde{\Phi}}=I\otimes(\widetilde{A}-\widetilde{\Phi} \widetilde{C})+(\widetilde{D}-\widetilde{C}\widetilde{\Phi})^T\otimes I$,
   $$\rho((I\otimes (A_1)_{\omega}+((D_1)_{\omega})^T\otimes I)^{-1}(I\otimes(|A_2|+\widetilde{\Phi} \widetilde{C})+(|D_2|+\widetilde{C}\widetilde{\Phi})^T\otimes I))<1.$$
Because
   \begin{eqnarray*}
   \begin{aligned}
   &~~~~\rho((I\otimes A_1+D_1^T\otimes I)^{-1}(I\otimes(A_2+\Phi C)+(D_2+C\Phi)^T\otimes I))\\
   &\leq\rho((I\otimes (A_1)_{\omega}+((D_1)_{\omega})^T\otimes I)^{-1}|I\otimes(|A_2|+|\Phi| |C|)+(|D_2|+|C||\Phi|)^T\otimes I|)\\
   &\leq\rho((I\otimes (A_1)_{\omega}+((D_1)_{\omega})^T\otimes I)^{-1}((I\otimes(|A_2|+\widetilde{\Phi} \widetilde{C})+(|D_2|+\widetilde{C}\widetilde{\Phi})^T\otimes I)))\\
   &<1,
   \end{aligned}
   \end{eqnarray*}
the sequence $\{\Phi_k\}$ is linearly convergent to the extremal solution $\Phi$.
\end{proof}

Let $A=H_A-L_A-U_A$ and $D=H_D-L_D-U_D$, where $H_A$ and $H_D$ are the diagonal parts, $-L_A$ and $-L_D$ are the strictly lower triangular parts, $-U_A$ and $-U_D$ are the strictly upper triangular parts, of $A$ and $D$, respectively. Then the following fourteen splittings for $A=A_1-A_2$ and $D=D_1-D_2$ obviously satisfy the conditions on splittings in Theorem \ref{thm8}:

(\uppercase\expandafter{\romannumeral1}) Trivial splitting:
 \begin{eqnarray*}
\begin{aligned}
&({\rm \romannumeral1})&A_1=A,~A_2=0,~D_1=D,~D_2=0,
\end{aligned}
\end{eqnarray*}
and we call the correspondingly resulted fixed-point iteration as trivial fixed-point (TFP) iteration.

(\uppercase\expandafter{\romannumeral2}) Jacobi-type splitting:
 \begin{eqnarray*}
\begin{aligned}
&({\rm \romannumeral2})&A_1=H_A,~A_2=L_A+U_A,~D_1=H_D,~D_2=L_D+U_D,
\end{aligned}
\end{eqnarray*}
and we call the correspondingly resulted fixed-point iteration as Jacobi-type fixed-point (JFP) iteration.

(\uppercase\expandafter{\romannumeral3}) Gauss-Seidel-type splittings:
\begin{eqnarray*}
\begin{aligned}
&({\rm \romannumeral3})&A_1=H_A-U_A,~A_2=L_A,~D_1=H_D-L_A,~D_2=U_D;\\
&({\rm \romannumeral4})&A_1=H_A-U_A,~A_2=L_A,~D_1=H_D-U_A,~D_2=L_D;\\
&({\rm \romannumeral5})&A_1=H_A-L_A,~A_2=U_A,~D_1=H_D-L_D,~D_2=U_D;\\
&({\rm \romannumeral6})&A_1=H_A-L_A,~A_2=U_A,~D_1=H_D-U_D,~D_2=L_D;\\
\end{aligned}
\end{eqnarray*}
and we call the correspondingly resulted fixed-point iterations as Gauss-Seidel-type fixed-point (GSFP) iterations.

(\uppercase\expandafter{\romannumeral4}) SOR-type splittings ($\omega$ is a positive parameter):
\begin{eqnarray*}
\begin{aligned}
&({\rm \romannumeral7})&A_1&=(1/\omega)H_A-U_A,&A_2&=(1/\omega-1)H_A+L_A,\\
&&D_1&=(1/\omega)H_D-L_D,&D_2&=(1/\omega-1)H_D+U_D;
\\
&({\rm \romannumeral8})&A_1&=(1/\omega)H_A-U_A,&A_2&=(1/\omega-1)H_A+L_A,\\ &&D_1&=(1/\omega)H_D-U_D,&D_2&=(1/\omega-1)H_D+L_D;\\
&({\rm \romannumeral9})&A_1&=(1/\omega)H_A-L_A,&A_2&=(1/\omega-1)H_A+U_A,\\ &&D_1&=(1/\omega)H_D-L_D,&D_2&=(1/\omega-1)H_D+U_D;\\
&({\rm \romannumeral10})&A_1&=(1/\omega)H_A-L_A,&A_2&=(1/\omega-1)H_A+U_A,\\
&&D_1&=(1/\omega)H_D-U_D,&D_2&=(1/\omega-1)H_D+L_D;
\end{aligned}
\end{eqnarray*}
and we call the correspondingly resulted fixed-point iterations as SOR-type fixed-point (SORFP) iterations.

(\uppercase\expandafter{\romannumeral5}) AOR-type splittings ($\omega$ and $\gamma$ are two positive parameters):
\begin{eqnarray*}
\begin{aligned}
&({\rm \romannumeral11})&A_1&=(1/\omega)(H_A-\gamma U_A),&A_2&=(1/\omega)((1-\omega)H_A+(\omega-\gamma)U_A+\omega L_A),\\
&&D_1&=(1/\omega)(H_D-\gamma L_D),&D_2&=(1/\omega)((1-\omega)H_D+(\omega-\gamma)L_D+\omega U_D);\\
&({\rm \romannumeral12})&A_1&=(1/\omega)(H_A-\gamma U_A),&A_2&=(1/\omega)((1-\omega)H_A+(\omega-\gamma)U_A+\omega L_A),\\
&&D_1&=(1/\omega)(H_D-\gamma U_D),&D_2&=(1/\omega)((1-\omega)H_D+(\omega-\gamma)U_D+\omega L_D);\\
&({\rm \romannumeral13})&A_1&=(1/\omega)(H_A-\gamma L_A),&A_2&=(1/\omega)((1-\omega)H_A+(\omega-\gamma)L_A+\omega U_A),\\
&&D_1&=(1/\omega)(H_D-\gamma U_D),&D_2&=(1/\omega)((1-\omega)H_D+(\omega-\gamma)U_D+\omega L_D);\\
&({\rm \romannumeral14})&A_1&=(1/\omega)(H_A-\gamma L_A),&A_2&=(1/\omega)((1-\omega)H_A+(\omega-\gamma)L_A+\omega U_A),\\
&&D_1&=(1/\omega)(H_D-\gamma L_D),&D_2&=(1/\omega)((1-\omega)H_D+(\omega-\gamma)L_D+\omega U_D);
\end{aligned}
\end{eqnarray*}
and we call the correspondingly resulted fixed-point iterations as AOR-type fixed-point (AORFP) iterations; Evidently, JFP, GSFP and SORFP can be considered as special cases of AORFP when the iteration parameters $(\omega, \gamma)$ are specified to be $(1,0), (1,1)$ and $(\omega,\omega)$, respectively.

\section{Apply doubling algorithms to solve the NARE (\ref{e1}) in class $\mathbb{H}^\omega$}
ADDA and SDA are two existing doubling algorithms which have been successfully applied to the NAREs in class $\mathbb{M}$, $\mathbb{H}^+$, $\mathbb{H}^-$ and $\mathbb{H}^*$. This section discusses how to choose suitable parameters to make them also deliver the extremal solution $\Phi$ of the NARE (\ref{e1}) in class $\mathbb{H}^\omega$, and analyzes the convergence of the two doubling algorithms.

\subsection{{\rm General framework of doubling algorithms for NAREs}}
Let $\Phi$ and $\Psi$ be the solutions of the NAREs (\ref{e1}) and (\ref{e3}), respectively. Then
   \begin{eqnarray}\label{e16}
     H\left(
     \begin{matrix}
     I\\
     \Phi
     \end{matrix}
     \right)
     =
     \left(
     \begin{matrix}
     I\\
     \Phi
     \end{matrix}
     \right)R,~~
     H\left(
     \begin{matrix}
     \Psi\\
     I
     \end{matrix}
     \right)=\left(
     \begin{matrix}
     \Psi\\
     I
     \end{matrix}
     \right)(-S),
   \end{eqnarray}
where
   \begin{eqnarray}\label{e17}
     R=D-C\Phi,~~S=A-B\Psi.
   \end{eqnarray}
First choose some suitable parameters and transform (\ref{e16}) into
   \begin{eqnarray}\label{e18}
     \left(
     \begin{matrix}
     E_0&0\\
     -H_0&I
     \end{matrix}
     \right)\left(
     \begin{matrix}
     I\\
     \Phi
     \end{matrix}
     \right)
     =
     \left(
     \begin{matrix}
     I&-G_0\\
     0&F_0
     \end{matrix}
     \right)\left(
     \begin{matrix}
     I\\
     \Phi
     \end{matrix}
     \right)\mathcal{R},~~
     \left(
     \begin{matrix}
     E_0&0\\
     -H_0&I
     \end{matrix}
     \right)\left(
     \begin{matrix}
     \Psi\\
     I
     \end{matrix}
     \right)\mathcal{S}
     =
     \left(
     \begin{matrix}
     I&-G_0\\
     0&F_0
     \end{matrix}
     \right)\left(
     \begin{matrix}
     \Psi\\
     I
     \end{matrix}
     \right).
   \end{eqnarray}
Then construct a sequence of quaternion $\{E_k,F_k,G_k,H_k\},k=0,1,2,\cdots,$ such that
   \begin{eqnarray}\label{e19}
     \left(
     \begin{matrix}
     E_k&0\\
     -H_k&I
     \end{matrix}
     \right)\left(
     \begin{matrix}
     I\\
     \Phi
     \end{matrix}
     \right)=\left(
     \begin{matrix}
     I&-G_k\\
     0&F_k
     \end{matrix}
     \right)
     \left(
     \begin{matrix}
     I\\
     \Phi
     \end{matrix}
     \right)
     \mathcal{R}^{2^k},~
     \left(
     \begin{matrix}
     E_k&0\\
     -H_k&I
     \end{matrix}
     \right)
     \left(
     \begin{matrix}
     \Psi\\
     I
     \end{matrix}
     \right)
     \mathcal{S}^{2^k}
     =
     \left(
     \begin{matrix}
     I&-G_k\\
     0&F_k
     \end{matrix}
     \right)\left(
     \begin{matrix}
     \Psi\\
     I
     \end{matrix}
     \right).
   \end{eqnarray}
It can be verified that
   \begin{eqnarray}\label{e20}
     \Phi-H_k=(I-H_k\Psi)\mathcal{S}^{2^k}\Phi \mathcal{R}^{2^k},~ \Psi-G_k=(I-G_k\Phi)\mathcal{R}^{2^k}\Psi \mathcal{S}^{2^k}.
   \end{eqnarray}
If $H_k$ and $G_k$ are uniformly bounded with respect to $k$, then it was shown by \cite{Axe14} that for any consistent matrix norm $\parallel\cdot\parallel$,
   \begin{eqnarray*}
     \begin{aligned}
      &\lim_{k\rightarrow \infty}{\rm sup}\parallel\Phi-H_k\parallel^{1/2^k}\leq \rho(\mathcal{R})\rho(\mathcal{S}),\\
      &\lim_{k\rightarrow \infty}{\rm sup} \parallel\Psi-G_k\parallel^{1/2^k}\leq \rho(\mathcal{R})\rho(\mathcal{S}),
     \end{aligned}
   \end{eqnarray*}
which implies the sequences $\{H_k\}$ and $\{G_k\}$ converge quadratically to $\Phi$ and $\Psi$, respectively, if
   \begin{eqnarray}\label{e21}
     \rho(\mathcal{R})\rho(\mathcal{S})<1.
   \end{eqnarray}

\subsection{Doubling algorithms for the NARE (\ref{e7}) in class $\mathbb{M}$}

In this subsection, for subsequent convenience, we give a iterative scheme for the NARE (\ref{e7}) in class $\mathbb{M}$ which is slightly different from that proposed by \cite{Axe1}. In order to distinguish from the doubling algorithms for the NARE (\ref{e1}) in class $\mathbb{H}^\omega$, we use different notations to denote all involved matrices in ADDA, such as $\widetilde{\mathcal{R}}$ instead of $\mathcal{R}$.

Let $\widetilde{R}=\widetilde{D}-\widetilde{C}\widetilde{\Phi}$ and $\widetilde{S}=\widetilde{A}-\widetilde{B}\widetilde{\Psi}$, we first choose suitable complex parameters $\widetilde{\alpha}$, $\widetilde{\beta}$ such that
   $$\widetilde{R}+(\omega {\rm Re}(\widetilde{\alpha})+(1-\omega){\rm Im}(\widetilde{\alpha})) I,~~\widetilde{S}+(\omega {\rm Re}(\widetilde{\beta})+(1-\omega){\rm Im}(\widetilde{\beta})) I,$$
   $$\widetilde{A}_{\widetilde{\beta}}=\widetilde{A}+(\omega {\rm Re}(\widetilde{\beta})+(1-\omega){\rm Im}(\widetilde{\beta})) I,~~\widetilde{D}_{\widetilde{\alpha}}=\widetilde{D}+(\omega {\rm Re}(\widetilde{\alpha})+(1-\omega){\rm Im}(\widetilde{\alpha})) I,$$
   and
   $$\widetilde{W}_{\widetilde{\alpha} \widetilde{\beta}}=\widetilde{A}_{\widetilde{\beta}}-\widetilde{B}\widetilde{D}_{\widetilde{\alpha}}^{-1}\widetilde{C},~~~\widetilde{V}_{\widetilde{\alpha} \widetilde{\beta}}=\widetilde{D}_{\widetilde{\alpha}}-\widetilde{C}\widetilde{A}_{\widetilde{\beta}}^{-1}\widetilde{B}$$
are nonsingular. Then the constructions of  the matrices $\widetilde{\mathcal{R}}$ and $\widetilde{\mathcal{S}}$ and initial matrices $\widetilde{E}_0$, $\widetilde{F}_0$, $\widetilde{G}_0$, $\widetilde{H}_0$ in ADDA are as follows.
   \begin{eqnarray}\label{e22}
     \begin{aligned}
     \widetilde{\mathcal{R}}&=(\widetilde{R}-(\omega {\rm Re}(\widetilde{\beta})+(1-\omega){\rm Im}(\widetilde{\beta})) I)(\widetilde{R}+(\omega {\rm Re}(\widetilde{\alpha})+(1-\omega){\rm Im}(\widetilde{\alpha})) I)^{-1},\\
     \widetilde{ \mathcal{S}}&=(\widetilde{S}-(\omega {\rm Re}(\widetilde{\alpha})+(1-\omega){\rm Im}(\widetilde{\alpha})) I)(\widetilde{S}+(\omega {\rm Re}(\widetilde{\beta})+(1-\omega){\rm Im}(\widetilde{\beta})) I)^{-1},\\
     \widetilde{E}_0&=I-((\omega {\rm Re}(\widetilde{\alpha})+(1-\omega){\rm Im}(\widetilde{\alpha}))+(\omega {\rm Re}(\widetilde{\beta})+(1-\omega){\rm Im}(\widetilde{\beta})))\widetilde{V}_{\widetilde{\alpha}\widetilde{\beta}}^{-1},\\
     \widetilde{F}_0&=I-((\omega {\rm Re}(\widetilde{\alpha})+(1-\omega){\rm Im}(\widetilde{\alpha}))+(\omega {\rm Re}(\widetilde{\beta})+(1-\omega){\rm Im}(\widetilde{\beta})))\widetilde{W}_{\widetilde{\alpha}\widetilde{\beta}}^{-1},\\
     \widetilde{G}_0&=((\omega {\rm Re}(\widetilde{\alpha})+(1-\omega){\rm Im}(\widetilde{\alpha}))+(\omega {\rm Re}(\widetilde{\beta})+(1-\omega){\rm Im}(\widetilde{\beta})))\widetilde{D}_{\widetilde{\alpha}}^{-1}\widetilde{C}\widetilde{W}_{\widetilde{\alpha}\widetilde{\beta}}^{-1},\\
     \widetilde{H}_0&=((\omega {\rm Re}(\widetilde{\alpha})+(1-\omega){\rm Im}(\widetilde{\alpha}))+(\omega {\rm Re}(\widetilde{\beta})+(1-\omega){\rm Im}(\widetilde{\beta})))\widetilde{W}_{\widetilde{\alpha}\widetilde{\beta}}^{-1}\widetilde{B}\widetilde{D}_{\widetilde{\alpha}}^{-1}.
    \end{aligned}
   \end{eqnarray}
After $\widetilde{E}_0$, $\widetilde{F}_0$, $\widetilde{G}_0$ and $\widetilde{H}_0$ are set, the sequence $\{\widetilde{E}_k,\widetilde{F}_k,\widetilde{G}_k,\widetilde{H}_k\},k=0,1,2,\cdots,$ is produced with the iterative scheme:
\begin{eqnarray}\label{e23}
     \begin{aligned}
      \widetilde{E}_{k+1}&=\widetilde{E}_k(I-\widetilde{G}_k\widetilde{H}_k)^{-1}\widetilde{E}_k,\\
      \widetilde{F}_{k+1}&=\widetilde{F}_k(I-\widetilde{H}_k\widetilde{G}_k)^{-1}\widetilde{F}_k,\\
      \widetilde{G}_{k+1}&=\widetilde{G}_k+\widetilde{E}_k(I-\widetilde{G}_k\widetilde{H}_k)^{-1}\widetilde{G}_k\widetilde{F}_k,\\
      \widetilde{H}_{k+1}&=\widetilde{H}_k+\widetilde{F}_k(I-\widetilde{H}_k\widetilde{G}_k)^{-1}\widetilde{H}_k\widetilde{E}_k,
     \end{aligned}
\end{eqnarray}
as long as the matrices $I-\widetilde{G}_k\widetilde{H}_k$ and $I-\widetilde{H}_k\widetilde{G}_k$ are invertible for all $k$.

\begin{theorem}\label{thm10}
{\rm (\cite[see.][]{Axe9}.} Let $\widetilde{Q}$ in {\rm(\ref{e8})} be a nonsingular M-matrix. Let $\widetilde{\Phi}$ and $\widetilde{\Psi}$ be the minimal nonnegative solutions of the NARE {\rm(\ref{e7})} and its dual NARE {\rm(\ref{e9})}, respectively. Let the sequences $\{\widetilde{E}_k\}$, $\{\widetilde{F}_k\}$, $\{\widetilde{G}_k\}$ and $\{\widetilde{H}_k\}$ be generated by ADDA applied to the NARE {\rm(\ref{e7})} with the parameters $\widetilde{\alpha}$, $\widetilde{\beta}$ satisfying
   $$\omega {\rm Re}(\widetilde{\alpha})+(1-\omega){\rm Im}(\widetilde{\alpha}) \geq \max_{i=1,2,\cdots,m}[\widetilde{A}]_{ii}=\max_{i=n+1,n+2,\cdots,m+n}[\widetilde{Q}]_{ii},$$
   $$\omega {\rm Re}(\widetilde{\beta})+(1-\omega){\rm Im}(\widetilde{\beta}) \geq \max_{i=1,2,\cdots,n}[\widetilde{D}]_{ii}=\max_{i=1,2,\cdots,n}[\widetilde{Q}]_{ii},$$
then

{\rm(1)} $\rho(\widetilde{\mathcal{R}})\rho(\widetilde{\mathcal{S}})<1$;

{\rm(2)} $\widetilde{E}_0$, $\widetilde{F}_0\leq 0$; $\widetilde{E}_k$, $\widetilde{F}_k\geq 0$, for $k\geq 1$;

{\rm(3)} $I-\widetilde{G}_k\widetilde{H}_k$ and $I-\widetilde{H}_k\widetilde{G}_k$ are nonsingular M-matrices;

{\rm(4)} $0\leq \widetilde{H}_k\leq \widetilde{H}_{k+1}\leq \widetilde{\Phi}$, $0\leq \widetilde{G}_k\leq \widetilde{G}_{k+1}\leq \widetilde{\Psi}$, and
   $$\lim_{k\rightarrow \infty}\widetilde{H}_k=\widetilde{\Phi},~~~\lim_{k\rightarrow\infty}\widetilde{G}_{k}=\widetilde{\Psi}.$$
\end{theorem}

Notice that ADDA is reduced to SDA when $\widetilde{\alpha}=\widetilde{\beta}$.

\subsection{Doubling algorithms for the NARE (\ref{e1}) in class $\mathbb{H}^\omega$}

We choose suitable complex parameters $\alpha$, $\beta$ such that
   $$R+\alpha I,~~S+\beta I,~~A_{\beta}=A+\beta I,~~D_{\alpha}=D+\alpha I,$$
   and
   $$W_{\alpha \beta}=A_{\beta}-BD_{\alpha}^{-1}C,~~~V_{\alpha \beta}=D_{\alpha}-CA_{\beta}^{-1}B$$
are nonsingular. Then the constructions of $\mathcal{R}$, $\mathcal{S}$, $E_0$, $F_0$, $G_0$ and $H_0$ in ADDA are as follows.
   \begin{eqnarray}\label{e24}
     \begin{aligned}
     \mathcal{R}&=(R-\beta I)(R+\alpha I)^{-1},& \mathcal{S}&=(S-\alpha I)(S+\beta I)^{-1},\\
     E_0&=I-(\alpha+\beta)V_{\alpha\beta}^{-1},&F_0&=I-(\alpha+\beta)W_{\alpha\beta}^{-1},\\
     G_0&=(\alpha+\beta)D_{\alpha}^{-1}CW_{\alpha\beta}^{-1},&H_0&=(\alpha+\beta)W_{\alpha\beta}^{-1}BD_{\alpha}^{-1}.
    \end{aligned}
   \end{eqnarray}
After $E_0,F_0,G_0$ and $H_0$ are set, the sequence $\{E_k,F_k,G_k,H_k\},k=0,1,2,\cdots,$ is produced with the iterative scheme:
   \begin{eqnarray}\label{e23}
     \begin{aligned}
      E_{k+1}&=E_k(I-G_kH_k)^{-1}E_k,\\
      F_{k+1}&=F_k(I-H_kG_k)^{-1}F_k,\\
      G_{k+1}&=G_k+E_k(I-G_kH_k)^{-1}G_kF_k,\\
      H_{k+1}&=H_k+F_k(I-H_kG_k)^{-1}H_kE_k,
     \end{aligned}
   \end{eqnarray}
as long as the matrices $I-G_kH_k$ and $I-H_kG_k$ are invertible for all $k$.

\begin{lemma}\label{thm11}
Let $A\in \mathbb{R}^{n\times n}$ be a nonsingular M-matrix and be written as $A=D_1-N_1$, where $D_{1}$ is diagonal with positive diagonal entries and $N_1\geq 0$. Let $B\in \mathbb{C}^{n\times n}$ be written as $B=D_2-N_2$, where $D_2$ is diagonal. If
$$D_1\leq |D_2|,~~~|N_2|\leq N_1,$$
then $B$ is nonsingular with $|B^{-1}|\leq A^{-1}C$ and $|B^{-1}|\leq CA^{-1}$, where $C=|D_1D_2^{-1}|$.
\end{lemma}

\begin{theorem}\label{thm12}
Let $\Phi$ and $\Psi$ be the extremal solutions of the NARE {\rm(\ref{e1})} and its dual NARE {\rm(\ref{e3})} in class $\mathbb{H}^\omega$, respectively. Suppose that $Q_{\omega}\textbf{1} > 0$. Let $\{E_k\}$, $\{F_k\}$, $\{G_k\}$, $\{H_k\}$ be generated by ADDA applied to the NARE {\rm(\ref{e1})} with the
parameters $\alpha$, $\beta$ ($\omega {\rm Re}(\alpha)+(1-\omega){\rm Im}(\alpha)>0$, $\omega {\rm Re}(\beta)+(1-\omega){\rm Im}(\beta)>0$) satisfying
   \begin{eqnarray}
     &\frac{\omega {\rm Re}(\alpha)+(1-\omega) {\rm Im}(\alpha)+q_i}{|\alpha+[Q]_{ii}|}|\beta-[Q]_{ii}|<\omega {\rm Re}(\beta)+(1-\omega) {\rm Im}(\beta)-q_i,~i=1,2,\cdots,n;~~\label{e25}\\
     &\frac{\omega {\rm Re}(\beta)+(1-\omega) {\rm Im}(\beta)+q_i}{|\beta+[Q]_{ii}|}|\alpha-[Q]_{ii}|<\omega {\rm Re}(\alpha)+(1-\omega) {\rm Im}(\alpha)-q_i,~i=n+1,n+2,\cdots,n+m.~~~~\label{e26}
   \end{eqnarray}
Then $\{H_k\}$ and $\{G_k\}$ quadratically converge to $\Phi$ and $\Psi$, respectively.
\end{theorem}

\begin{proof}

Since $Q_{\omega}\textbf{1}>0$, we can take $\epsilon>0$ sufficiently small such that
$$\omega {\rm Re}([Q]_{ii})+(1-\omega){\rm Im}([Q]_{ii})\geq q_i+\epsilon,~~~~i=1,2,\cdots,m+n.$$ Furthermore, we assume that $\epsilon$ is small enough such that the inequalities in (\ref{e25}) and (\ref{e26}) strictly hold with $q_i$ and $q_j$ replaced by $q_i+\epsilon$ and $q_j+\epsilon$, respectively.
Let $\widetilde{Q}=
    \begin{cases}
    q_i+\epsilon,&i=j,\\
    -|[Q]_{ij}|,&i\neq j,
    \end{cases}$
and $\{\widetilde{E}_{k}\}$, $\{\widetilde{F}_{k}\}$, $\{\widetilde{G}_{k}\}$ and $\{\widetilde{H}_{k}\}$ be generated by ADDA applied to the MARE (\ref{e7}) with parameters $\widetilde{\alpha}=\alpha$, $\widetilde{\beta}=\beta$. As $(\omega {\rm Re}(\widetilde{\alpha})+(1-\omega){\rm Im}(\widetilde{\alpha}))\geq \max_{i=n+1,n+2,\cdots,n+m}[\widetilde{Q}]_{ii}$ and $(\omega {\rm Re}(\widetilde{\beta})+(1-\omega){\rm Im}(\widetilde{\beta}))\geq \max_{i=1,2,\cdots, n}[\widetilde{Q}]_{ii}$, Theorem \ref{thm10} applies.

We need to show that ADDA is well-defined when applied to the NARE (\ref{e1}) in class $\mathbb{H}^\omega$ and that the sequences $\{H_k\}$ and $\{G_k\}$ are bounded.  Similarly to the statements explained by \cite{Axe1}, we observe that it is enough to show
\begin{eqnarray}\label{e27}
\begin{aligned}
|E_0|&\leq |\widetilde{E}_0|,&|F_0|&\leq |\widetilde{F}_0|,\\
|G_0|&\leq \widetilde{G}_0,&|H_0|&\leq \widetilde{H}_0,
\end{aligned}
\end{eqnarray}
where $E_0$, $F_0$, $G_0$ and $H_0$ are as in (\ref{e24}).

Let
$$N={\rm diag}(\widetilde{D}_{\widetilde{\alpha}})|{\rm diag}(\alpha I+D)|^{-1},~~~M={\rm diag}(\widetilde{A}_{\widetilde{\beta}})|{\rm diag}(\beta I+A)|^{-1}.$$
Because $\widetilde{A}\leq A_{\omega}$ and $\widetilde{D}\leq D_{\omega}$,
\begin{eqnarray*}
\begin{aligned}
{\rm diag}(\widetilde{D}_{\widetilde{\alpha}})&={\rm diag}(\widetilde{D})+\omega {\rm Re}(\widetilde{\alpha})I+(1-\omega){\rm Im}(\widetilde{\alpha}) I\\
&\leq\omega {\rm diag}({\rm Re}(D))+(1-\omega){\rm diag}({\rm Im}(D))+\omega {\rm Re}(\alpha)I+(1-\omega){\rm Im}(\alpha) I\\
&=\omega({\rm diag}({\rm Re}(D))+{\rm Re}(\alpha) I)+(1-\omega)({\rm diag}({\rm Im}(D))+{\rm Im}(\alpha) I)\\
&\leq \sqrt{({\rm diag}({\rm Re}(D))+{\rm Re}(\alpha) I)^2+({\rm diag}({\rm Im}(D))+{\rm Im}(\alpha) I)^2}\\
&=|{\rm diag}(D+\alpha I)|=|{\rm diag}(D_{\alpha})|.
\end{aligned}
\end{eqnarray*}
Then $0\leq N\leq I_n$, $0\leq M\leq I_m$, and the inequalities (\ref{e25}) and (\ref{e26}) in Theorem \ref{thm12} can be written as
\begin{eqnarray}
N|{\rm diag}(\beta I-D)|\leq {\rm diag}((\omega {\rm Re}(\widetilde{\beta})+(1-\omega){\rm Im}(\widetilde{\beta})) I-\widetilde{D}),\label{e28}\\
M|{\rm diag}(\alpha I-A)|\leq {\rm diag}((\omega {\rm Re}(\widetilde{\alpha})+(1-\omega){\rm Im}(\widetilde{\alpha})) I-\widetilde{A}).\label{e29}
\end{eqnarray}
Because $\widetilde{A}_{\widetilde{\beta}},\widetilde{D}_{\widetilde{\alpha}},\widetilde{W}_{\widetilde{\alpha}\widetilde{\beta}}$ and $\widetilde{V}_{\widetilde{\alpha}\widetilde{\beta}}$ are nonsingular M-matrices, it follows from Lemma \ref{thm11} that
$$|(A_{\beta})^{-1}|\leq (\widetilde{A}_{\widetilde{\beta}})^{-1}M\leq (\widetilde{A}_{\widetilde{\beta}})^{-1},~~~~|(D_{\alpha})^{-1}|\leq (\widetilde{D}_{\widetilde{\alpha}})^{-1}N\leq (\widetilde{D}_{\widetilde{\alpha}})^{-1}.$$
Thus
\begin{eqnarray*}
\begin{aligned}
|W_{\alpha\beta}^{-1}|&=|(A_{\beta}-BD_{\alpha}^{-1}C)^{-1}|\\
&\leq |(I-A_{\beta}^{-1}BD_{\alpha}^{-1}C)^{-1}||A_{\beta}^{-1}|\\
&\leq (I-\widetilde{A}_{\widetilde{\beta}}^{-1}\widetilde{B}\widetilde{D}_{\widetilde{\alpha}}^{-1}\widetilde{C})^{-1}\widetilde{A}_{\widetilde{\beta}}^{-1}M\\
&=\widetilde{W}_{\widetilde{\alpha}\widetilde{\beta}}^{-1}M\leq \widetilde{W}_{\widetilde{\alpha}\widetilde{\beta}}^{-1}.
\end{aligned}
\end{eqnarray*}
Similarly,
$$|V_{\alpha\beta}^{-1}|\leq \widetilde{V}_{\widetilde{\alpha}\widetilde{\beta}}^{-1}N\leq\widetilde{V}_{\widetilde{\alpha}\widetilde{\beta}}^{-1} .$$

From the inequality (\ref{e28}), we have
\begin{eqnarray*}
\begin{aligned}
|E_0|&=|V_{\alpha\beta}^{-1}(D-\beta I-CA_{\beta}^{-1}B)|\\
&\leq  \widetilde{V}_{\widetilde{\alpha}\widetilde{\beta}}^{-1}N(|\beta I-D|+|CA_{\beta}^{-1}B|)\\
&\leq \widetilde{V}_{\widetilde{\alpha}\widetilde{\beta}}^{-1}(N|\beta I-D|+\widetilde{C}\widetilde{A}_{\widetilde{\beta}}^{-1}\widetilde{B})\\
&\leq \widetilde{V}_{\widetilde{\alpha}\widetilde{\beta}}^{-1}((\omega {\rm Re}(\widetilde{\beta})+(1-\omega){\rm Im}(\widetilde{\beta})) I-\widetilde{D}+\widetilde{C}\widetilde{A}_{\widetilde{\beta}}^{-1}\widetilde{B})\\
&=-\widetilde{E}_0=|\widetilde{E}_0|
\end{aligned}
\end{eqnarray*}
Similarly, $|F_0|\leq -\widetilde{F}_0=|\widetilde{F}_0|$ by (\ref{e29}).
It follows from (\ref{e24}) that
\begin{eqnarray*}
\begin{aligned}
|G_0|&= |(\alpha+\beta)D_{\alpha}^{-1}CW_{\alpha\beta}^{-1}|\\
&=|D_{\alpha}^{-1}C(I-F_0)|\\
&\leq|D_{\alpha}^{-1}||C||I-F_0|\\
&\leq\widetilde{D}_{\widetilde{\alpha}}^{-1}\widetilde{C}(I+|F_0|)\\
&\leq \widetilde{D}_{\widetilde{\alpha}}^{-1}\widetilde{C}(I-\widetilde{F}_0)\\
&=(\omega {\rm Re}(\widetilde{\alpha}+\widetilde{\beta})+(1-\omega){\rm Im}(\widetilde{\alpha}+\widetilde{\beta}))\widetilde{D}_{\widetilde{\alpha}}^{-1}\widetilde{C}\widetilde{W}_{\widetilde{\alpha}\widetilde{\beta}}^{-1}\\
&=\widetilde{G}_0
\end{aligned}
\end{eqnarray*}
Similarly, $|H_0|\leq \widetilde{H}_0$.  To complete the proof, we also need to show that  $\rho(\mathcal{R})\rho(\mathcal{S})<1$. We have
$$|(R+\alpha I)^{-1}|\leq (\widetilde{R}+(\omega {\rm Re}(\widetilde{\alpha})+(1-\omega){\rm Im}(\widetilde{\alpha})) I)^{-1}N$$
and
$$|R-\beta I|\leq -N^{-1}(\widetilde{R}-(\omega {\rm Re}(\widetilde{\beta})+(1-\omega){\rm Im}(\widetilde{\beta})) I).$$
Then
\begin{eqnarray*}
\begin{aligned}
|\mathcal{R}|&\leq |R-\beta I||(R+\alpha I)^{-1}|\leq \\
&-N^{-1}(\widetilde{R}-(\omega {\rm Re}(\widetilde{\beta})+(1-\omega){\rm Im}(\widetilde{\beta})) I)(\widetilde{R}+(\omega {\rm Re}(\widetilde{\alpha})+(1-\omega){\rm Im}(\widetilde{\alpha})) I)^{-1}N\\
&=-N^{-1}\widetilde{\mathcal{R}}N.
\end{aligned}
\end{eqnarray*}
Similarly, $|\mathcal{S}|\leq-M^{-1}\mathcal{\widetilde{S}}M$. Therefore $\rho(\mathcal{R})\rho(\mathcal{S})\leq\rho(\widetilde{\mathcal{R}})\rho(\widetilde{\mathcal{S}})<1$.
\end{proof}

Notice that ADDA is reduced to SDA when $\alpha=\beta$. We will gain Theorem \ref{thm13} below which states the convergence of SDA applied to the NARE (\ref{e1}) in class $\mathbb{H}^\omega$.

\begin{theorem}\label{thm13}
Let $\Phi$ and $\Psi$ be the extremal solutions of the NARE {\rm(\ref{e1})} and its dual NARE {\rm(\ref{e3})} in class $\mathbb{H}^\omega$, respectively. Suppose that $Q_{\omega}\textbf{1} > 0$. Let $\{E_k\}$, $\{F_k\}$, $\{G_k\}$ and $\{H_k\}$ be generated by SDA applied to the NARE {\rm(\ref{e1})} with the
parameters $\alpha$ satisfying
   \begin{eqnarray*}
     \frac{\omega {\rm Re}(\alpha)+(1-\omega) {\rm Im}(\alpha)+q_i}{|\alpha+[Q]_{ii}|}|\alpha-[Q]_{ii}|<\omega {\rm Re}(\alpha)+(1-\omega) {\rm Im}(\alpha)-q_i,~~~~i=1,2,\cdots,m+n.
   \end{eqnarray*}
Then the sequences $\{H_k\}$ and $\{G_k\}$ quadratically converge to $\Phi$ and $\Psi$, respectively.
\end{theorem}

\subsection{Strategies of choosing parameters $\alpha$ and $\beta$}

Theoretically, we have known from (\ref{e25}) and (\ref{e26}) how to choose the parameters $\alpha$ and $\beta$, but the ranges of the parameters $\alpha$ and $\beta$ are not intuitive. We need to clarify them in terms of the real parts and imaginary parts of the parameters $\alpha$ and $\beta$. Next, we will search for the parameters $\alpha$ and $\beta$ in the ray $R(z_{\omega}^\bot)$.

Let $\alpha=tz_{\omega}^\bot$ and $\beta=\gamma z_{\omega}^\bot$ with $t>0$, $\gamma>0$. Then we have ${\rm Re}(\alpha)=t\omega$, ${\rm Im}(\alpha)=t(1-\omega)$ and ${\rm Re}(\beta)=\gamma\omega$, ${\rm Im}(\beta)=\gamma(1-\omega)$. We can rewrite (\ref{e25}) and (\ref{e26}) as
\begin{eqnarray*}
   \begin{aligned}
     \frac{t\omega^2+t(1-\omega)^2+q_i}{|\alpha+[Q]_{ii}|}|\beta-[Q]_{ii}|&<\gamma\omega^2+\gamma(1-\omega)^2-q_i,&i&=1,2,\cdots,n.\\
     \frac{\gamma\omega^2+\gamma(1-\omega)^2+q_i}{|\beta+[Q]_{ii}|}|\alpha-[Q]_{ii}|&<t\omega^2+t(1-\omega)^2-q_i, &i&=n+1,n+2,\cdots,n+m.
     \end{aligned}
\end{eqnarray*}
Let $\varpi=\omega^2+(1-\omega)^2$ with $\omega\in[0,1]$, then we have $\varpi>0$ and
\begin{eqnarray*}
\begin{aligned}
&(t^2\varpi^2+q_i^2+2t\varpi q_i)((\gamma\omega-{\rm Re}([Q]_{ii}))^2+(\gamma(1-\omega)-{\rm Im}([Q]_{ii}))^2)&&\\
<&(\gamma^2\varpi^2+q_i^2-2\gamma \varpi q_i)((t\omega+{\rm Re}([Q]_{ii}))^2+(t(1-\omega)+{\rm Im}([Q]_{ii}))^2),~~~\gamma \varpi>q_i,~~~i=1,2,\cdots,n.\\
&(\gamma ^2\varpi^2+q_i^2+2\gamma \varpi q_i)((t\omega-{\rm Re}([Q]_{ii}))^2+(t(1-\omega)-{\rm Im}([Q]_{ii}))^2)&&\\
<&(t^2\varpi^2+q_j^2-2t\varpi q_i)((\gamma \omega+{\rm Re}([Q]_{ii}))^2+(\gamma (1-\omega)+{\rm Im}([Q]_{ii}))^2),~~~t\varpi>q_i,~~~i=n+1,n+2,\cdots,n+m.
\end{aligned}
\end{eqnarray*}
The formulas are equivalent to
\begin{eqnarray*}
\begin{aligned}
(\varpi|[Q]_{ii}|^2-q_i^2)(t\varpi-\gamma \varpi)<&2(t\gamma \varpi^2-q_i(\omega {\rm Re}([Q]_{ii})+(1-\omega){\rm Im}([Q]_{ii})))(\omega {\rm Re}([Q]_{ii})+(1-\omega){\rm Im}([Q]_{ii})-q_i)\\
&-2q_i(\omega {\rm Im}([Q]_{ii})-(1-\omega){\rm Re}([Q]_{ii}))^2,~~~~\gamma \varpi>q_i,~~~~i=1,2,\cdots,n.\\
(\varpi|[Q]_{ii}|^2-q_j^2)(\gamma \varpi-t\varpi)<&2(t\gamma \varpi^2-q_i(\omega {\rm Re}([Q]_{ii})+(1-\omega){\rm Im}([Q]_{ii})))(\omega {\rm Re}([Q]_{ii})+(1-\omega){\rm Im}([Q]_{ii})-q_i)\\
&-2q_i(\omega {\rm Im}([Q]_{ii})-(1-\omega){\rm Re}([Q]_{ii}))^2,~~~~t\varpi>q_i,~~~~i=n+1,n+2,\cdots,n+m.
\end{aligned}
\end{eqnarray*}
Let
$$p_{\omega,i}=\frac{\omega {\rm Re}([Q]_{ii})+(1-\omega){\rm Im}([Q]_{ii})+q_i}{2}+\frac{(\omega{\rm Im}([Q]_{ii})-(1-\omega){\rm Re}([Q]_{ii}))^2}{2(\omega {\rm Re}([Q]_{ii})+(1-\omega){\rm Im}([Q]_{ii})-q_i)},~~~~i=1,2,\cdots,n+m,$$
$$s_{\omega,i}=\frac{\omega {\rm Re}([Q]_{ii})+(1-\omega){\rm Im}([Q]_{ii})-q_i}{2}+\frac{(\omega{\rm Im}([Q]_{ii})-(1-\omega){\rm Re}([Q]_{ii}))^2}{2(\omega {\rm Re}([Q]_{ii})+(1-\omega){\rm Im}([Q]_{ii})-q_i)},~~~~i=1,2,\cdots,n+m.$$
Then we have
\begin{eqnarray*}
\begin{aligned}
(t\varpi+p_{\omega,i})(\gamma \varpi-p_{\omega,i})&>-s_{\omega,i}^2,&\gamma \varpi&>q_i,&i&=1,2,\cdots,n,\\
(\gamma \varpi+p_{\omega,i})(t\varpi-p_{\omega,i})&>-s_{\omega,i}^2,&t\varpi&>q_i,&j&=n+1,n+2,\cdots,n+m,
\end{aligned}
\end{eqnarray*}
equivalently,
\begin{align}
&(t+\frac{p_{\omega,i}}{\varpi})(\gamma-\frac{p_{\omega,i}}{\varpi})>-\frac{s_{\omega,i}^2}{\varpi^2},~~~~\gamma>\frac{q_i}{\varpi}, ~~~~i=1,2,\cdots,n,,\label{e30}\\
&(\gamma+\frac{p_{\omega,i}}{\varpi})(t-\frac{p_{\omega,i}}{\varpi})>-\frac{s_{\omega,i}^2}{\varpi^2},~~~~t>\frac{q_i}{\varpi},~~~~j=n+1,n+2,\cdots,n+m,\label{e31}
\end{align}

\subsubsection{Immediate choice of the resulting parameters $t$ and $\gamma$}\

Let $$\psi_1=\max_{i=n+1,n+2,\cdots,m+n}\frac{p_{\omega,i}}{\varpi}=\frac{\max_{i=n+1,n+2,\cdots,m+n}p_{\omega,i}}{\varpi},~~~~\psi_2=\max_{i=1,2,\cdots,n}\frac{p_{\omega,i}}{\varpi}=\frac{\max_{i=1,2,\cdots,n}p_{\omega,i}}{\varpi}.$$
It is easy to show that $\frac{p_{\omega,i}}{\varpi}>\frac{q_i}{\varpi}$ since $\omega {\rm Re}([Q]_{ii})+(1-\omega){\rm Im}([Q]_{ii})>q_i$, $i=1,2,\cdots,m+n$. So, Theorem \ref{thm14} and Theorem \ref{thm15} below are obvious. Theorem \ref{thm14} states the convergence results for ADDA, and Theorem \ref{thm15} states the convergence results for SDA.

\begin{theorem}\label{thm14}
Suppose the NARE (\ref{e1}) is in class $\mathbb{H}^\omega$ and $Q_\omega \textbf{1}>0$. Let $\Phi$ and $\Psi$ be as in Theorem \ref{thm5}. Apply ADDA to the NARE (\ref{e1}) with $\alpha=tz_{\omega}^\bot$ and $\beta=\gamma z_{\omega}^\bot.$
If $$t\geq\psi_1,~~~~\gamma\geq\psi_2,$$ then the sequences $\{E_k\}$, $\{F_k\}$, $\{H_k\}$, $\{G_k\}$ are well-defined and $\{H_k\}$ and $\{G_k\}$ quadratically converge  to $\Phi$ and $\Psi$, respectively.
\end{theorem}

\begin{theorem}\label{thm15}
Suppose the NARE (\ref{e1}) is in class $\mathbb{H}^\omega$ and $Q_\omega \textbf{1}>0$. Let $\Phi$ and $\Psi$ be as in Theorem \ref{thm5}. Apply SDA to the NARE (\ref{e1}) with $\alpha=tz_{\omega}^\bot$.
If $$t\geq\max\{\psi_1,\psi_2\},$$ then the sequences $\{E_k\}$, $\{F_k\}$, $\{H_k\}$, $\{G_k\}$ are well-defined and $\{H_k\}$ and $\{G_k\}$ quadratically converge  to $\Phi$ and $\Psi$, respectively.
\end{theorem}

\subsubsection{Choice of the parameters $t$ and $\gamma$ using a preprocessing procedure}\

For a given NARE, $\frac{\max_{i=1,2,\cdots,m+n}p_{\omega,i}}{\varpi}$ may be very large, which will result in slow convergence. But it is quite possible that we can transform the given NARE into a new NARE for which $\frac{\max_{i=1,2,\cdots,m+n}p_{\omega,i}}{\varpi}$ is much smaller and the solution sets of the two NAREs are related in a simple way.

We first rewrite $\frac{p_{\omega,i}}{\varpi}$, $i=1,2,\cdots,m+n,$ in a more compact form:
$$\frac{p_{\omega,i}}{\varpi}=\frac{\varpi|[Q]_{ii}|^2-q_i^2}{2\varpi(\omega {\rm Re}([Q]_{ii})+(1-\omega){\rm Im}([Q]_{ii})-q_i)},~~i=1,2,\cdots,m+n,$$ where $\varpi=\omega^2+(1-\omega)^2$.

For the given NARE (\ref{e1}) with $Q_\omega \textbf{1}>0$ in class $\mathbb{H}^\omega$, we consider the new NARE
\begin{eqnarray}\label{e32}
X(\chi C)X-X(\chi D)-(\chi A)X+\chi B=0,
\end{eqnarray}
where $\chi$ is on the unit circle. Obviously, the new NARE (\ref{e32}) has the same solution set as the original one. The matrix corresponding to the NARE (\ref{e32}) is
\begin{eqnarray*}
Q_{\chi}=
\left(
\begin{matrix}
\chi D&-\chi C\\
-\chi B&\chi A
\end{matrix}
\right).
\end{eqnarray*}
Note that the $\omega$-comparison matrix $(Q_{\chi})_{\omega}$ of $Q_{\chi}$ is
\begin{eqnarray*}
(Q_{\chi})_{\omega}=
\left(
\begin{matrix}
(\chi D)_\omega&-|\chi C|\\
-|\chi B|&(\chi A)_\omega
\end{matrix}
\right),
\end{eqnarray*}
we have
\begin{eqnarray}\label{e33}
\begin{aligned}
\frac{p_{\omega,i}^\chi}{\varpi}&=\frac{\varpi|\chi [Q]_{ii}|^2-(|\chi| q_i)^2}{2\varpi(\omega {\rm Re}(\chi [Q]_{ii})+(1-\omega){\rm Im}(\chi [Q]_{ii})-|\chi|q_i)}\\
&=\frac{\varpi|[Q]_{ii}|^2-q_i^2}{2\varpi(\omega {\rm Re}(\chi [Q]_{ii})+(1-\omega){\rm Im}(\chi [Q]_{ii})-q_i)},~~~i=1,2,\cdots,m+n
\end{aligned}
\end{eqnarray}
where $\chi$ is such that $\omega {\rm Re}(\chi [Q]_{ii})+(1-\omega){\rm Im}(\chi [Q]_{ii})>q_i,i=1,2,\cdots,m+n$ (so the NARE(\ref{e32}) is still in class $\mathbb{H}^\omega$). Note that
\begin{eqnarray*}
\begin{aligned}
\varpi|[Q]_{ii}|^2-q_i^2&=(\omega {\rm Re}([Q]_{ii})+(1-\omega){\rm Im}([Q]_{ii})+q_i)(\omega {\rm Re}([Q]_{ii})+(1-\omega){\rm Im}([Q]_{ii})-q_i)\\
&~~~+(\omega{\rm Im}([Q]_{ii})-(1-\omega){\rm Re}([Q]_{ii}))^2>0\\
\end{aligned}
\end{eqnarray*}
since $q_i\geq0$ and $\omega {\rm Re}([Q]_{ii})+(1-\omega){\rm Im}([Q]_{ii})-q_i>0$ for $i=1,2,\cdots,m+n$. Thus, for each fixed $i$, in order to minimize the positive quantity $\frac{p_{\omega,i}^\chi}{\varpi}$, we need to search for a complex number $\chi$ such that $\omega {\rm Re}(\chi [Q]_{ii})+(1-\omega){\rm Im}(\chi [Q]_{ii})-q_i$ attains the maximal value.

Let $\chi=e^{-{\rm j}\vartheta}$ and $[Q]_{ii}=|[Q]_{ii}|e^{{\rm j}\phi_i},i=1,2,\cdots,m+n$. Then
\begin{eqnarray*}
\begin{aligned}
&~~~~\omega {\rm Re}(\chi [Q]_{ii})+(1-\omega){\rm Im}(\chi [Q]_{ii})-q_i\\
&=|[Q]_{ii}|(\omega {\rm Re}(e^{{\rm j}(\phi_i-\vartheta)})+(1-\omega){\rm Im}(e^{{\rm j}(\phi_i-\vartheta)}))-q_i\\
&=|[Q]_{ii}|(\omega {\rm cos}(\phi_i-\vartheta)+(1-\omega){\rm sin}(\phi_i-\vartheta))-q_i\\
&=|[Q]_{ii}|\sqrt{\omega^2+(1-\omega)^2}(\frac{\omega }{\sqrt{\omega^2+(1-\omega)^2}}{\rm cos}(\phi_i-\vartheta)+\frac{1-\omega }{\sqrt{\omega^2+(1-\omega)^2}}{\rm sin}(\phi_i-\vartheta))-q_i,\\
&~~~~~~~~~~~~~~~~~~~~~~~~~~~~~~~~~~~~~~~~~~~~~~~~~~~~~~~~~~~~~~~~~~~~~~i=1,2,\cdots,m+n.
\end{aligned}
\end{eqnarray*}
Let ${\rm cos}(\varphi)=\frac{\omega }{\sqrt{\omega^2+(1-\omega)^2}}$, and further ${\rm sin}(\varphi)=\frac{1-\omega }{\sqrt{\omega^2+(1-\omega)^2}}$, then for $i=1,2,\cdots,m+n,$
\begin{eqnarray*}
\begin{aligned}
\omega {\rm Re}(\chi [Q]_{ii})+(1-\omega){\rm Im}(\chi [Q]_{ii})-q_i&=|[Q]_{ii}|\sqrt{\omega^2+(1-\omega)^2}{\rm cos}(\phi_i-\varphi-\vartheta)-q_i,
\end{aligned}
\end{eqnarray*}
which attains the maximum $|[Q]_{ii}|\sqrt{\omega^2+(1-\omega)^2}-q_i>0$ at $\vartheta=\phi_i-\varphi$ since $\varpi|[Q]_{ii}|^2-q_i^2>0$.

If all complex numbers $[Q]_{ii}$ are on the same line passing through the origin, i.e., $\phi_i$ is constant for all $i$, then a common value $\chi=e^{-{\rm j}(\phi_i-\varphi)}$ will make all $\omega {\rm Re}(\chi [Q]_{ii})+(1-\omega){\rm Im}(\chi [Q]_{ii})-q_i$ attain the maximum $|[Q]_{ii}|\sqrt{\omega^2+(1-\omega)^2}-q_i>0$ for $i=1,2,\cdots,m+n$. In other words, the new NARE  with this $\chi$ is in class $\mathbb{H}^{\omega}$. In general, we cannot find a fixed $\chi=e^{-{\rm j}\vartheta}$ to minimize $\frac{p_{\omega,i}^\chi}{\varpi}$ for all $i$. So, we let
$$f_{i}(\vartheta)=\frac{\varpi|[Q]_{ii}|^2-q_i^2}{\varpi(|[Q]_{ii}|\sqrt{\omega^2+(1-\omega)^2}{\rm cos}(\phi_i-\varphi-\vartheta)-q_i)}, \vartheta\in \mathbb{R}, $$
and try to find $\vartheta$ such that
$$f(\vartheta)=\max_{i=1,2,\cdots,m+n}f_i(\vartheta)$$
is minimized, subject to the condition that $\omega {\rm Re}(e^{-{\rm j}\vartheta} [Q]_{ii})+(1-\omega){\rm Im}(e^{-{\rm j}\vartheta} [Q]_{ii})-q_i>0$ for all $i$.

\begin{theorem}\label{thm16}
Suppose the NARE (\ref{e1}) is in class $\mathbb{H}^\omega$ and $Q_{\omega}\textbf{1}>0$. Then the function $f(\vartheta)$ has a unique minimizer $\vartheta_*\in (-\pi,\pi)$ and $\vartheta_*$ can be computed by a bisection procedure.
\end{theorem}

\begin{proof}

The existence of a minimizer is quite obvious. Next we describe a unusual bisection procedure that can search for a unique minimizer. The procedure is based on the simple observation: Let $\phi_i$ be the angles of the complex numbers $[Q]_{ii}$, $i=1,2,\cdots,m+n$. Note that $\varphi-\frac{\pi}{2}<\phi_i<\varphi+\frac{\pi}{2}$ ($-\frac{\pi}{2}<\phi_i-\varphi<\frac{\pi}{2}$) and $f_i(\vartheta)$ is minimized at $\vartheta=\phi_i-\varphi$. Moreover, $f_i(\vartheta)$ is strictly decreasing on the left of $\phi_i-\varphi$ and is strictly increasing on the right of $\phi_i-\varphi$.

Let $d=\max_{i=1,2,\cdots,m+n}f_i(0)>0$. For each $i$ let $\Delta_i$ be
the set of all values $\vartheta\in (-\pi,\pi)$ such that $0<f_i(\vartheta)\leq d$. It is clear that $\Delta_i$ is a closed interval containing 0. Let $\Delta=\bigcap_{i=1,2,\cdots,m+n}\Delta_i$. Then $\Delta$ is also a closed interval containing 0. Now $\min f(\vartheta)=\min_{\vartheta\in \Delta} f(\vartheta)$ is attained at some $\vartheta_*\in \Delta$ since $f$ is continuous on $\Delta$.

The interval $\Delta$ above can be given explicitly as follows. Let $\underline{\delta_i}\leq \overline{\delta_i}$ be the two (usually different) solutions of $f_i(\vartheta)=d$. Namely, $\underline{\delta_i}=\phi_i-\varphi-\Psi_i$ and $\overline{\delta_i}=\phi_i-\varphi+\Psi_i$
with $$\Psi_i={\rm arccos}(\frac{1}{|[Q]_{ii}|\sqrt{\omega^2+(1-\omega)^2}}(q_i+\frac{\varpi|[Q]_{ii}|^2-q_i^2}{d\varpi})),~~~~ \Psi_i\in [0,\frac{\pi}{2}) $$
Now, $\Delta_i=[\underline{\delta_i},\overline{\delta_i}]$ and $\Delta=[\max \underline{\delta_i},\min \overline{\delta_i}]$. The interval $\Delta$ may be large even when all $\phi_i$ are equal to $\phi_*$, in which case we know that $\phi_*-\varphi$ is the unique minimizer of $f(\vartheta)$. To avoid using an unnecessarily large search interval in situations like this, we let $\phi_{\min}=\min \phi_i$ and $\phi_{\max}=\max \phi_i$ and we claim that any minimizer of $f(\vartheta)$ must be in $[\phi_{\min}-\varphi,\phi_{\max}-\varphi]$. In fact,
$$f_{i}(\vartheta)=\frac{\varpi|[Q]_{ii}|^2-q_i^2}{\varpi(\omega {\rm Re}(e^{-{\rm j}\vartheta} [Q]_{ii})+(1-\omega){\rm Im}(e^{-{\rm j}\vartheta} [Q]_{ii})-q_i)}=\frac{\varpi|[Q]_{ii}|^2-q_i^2}{\varpi(|[Q]_{ii}|\sqrt{\omega^2+(1-\omega)^2}{\rm cos}(\phi_i-\varphi-\vartheta)-q_i)}.$$
For any $\vartheta\in (-\pi, \phi_{\min}-\varphi)$ such that $f_i(\vartheta)>0$ for each $i$, we have $0<\phi_i-\phi_{\min}<\phi_i-\varphi-\vartheta<\pi+\phi_i-\varphi$ and $f_i(\vartheta)>f_{i}(\phi_{\min}-\varphi)>0$ for each $i$. So any minimizer $\vartheta_*$ must satisfy $\vartheta_*\geq \phi_{\min}-\varphi$. Similarly we can show that $\vartheta_*\leq \phi_{\max}-\varphi$.

The initial search interval for a minimizer of $f(\vartheta)$ is then $[\vartheta_{\min},\vartheta_{\max}]$, where
$$\vartheta_{\min}=\max\{\max\underline{\delta_i},\phi_{\min}-\varphi\},~~~~\vartheta_{\max}=\min\{\min\overline{\delta_i},\phi_{\max}-\varphi\}. $$
The first step of the bisection procedure is to take $\vartheta_1=\frac{1}{2}(\vartheta_{\min}+\vartheta_{\max})$. Let
$$a_1=\max_{\phi_i-\varphi>\vartheta_1}f_i(\vartheta_1),~~ b_1=\max_{\phi_i-\varphi<\vartheta_1}f_{i}(\vartheta_1),~~ c_1=\max_{\phi_i-\varphi=\vartheta_1}f_i(\vartheta_1),$$
where the maximum over an empty set is defined to be $0$. If $c_1\geq \max\{a_1,b_1\}$, then $\vartheta_1$ is a unique minimizer. Now suppose $c_1<\max\{a_1,b_1\}$. If $a_1=b_1$, then $\vartheta_1$ is still a unique minimizer. If $a_1\neq b_1$, any minimizer must be in $[\vartheta_1,\vartheta_{\max}]$ if $a_1>b_1$ and must be in $[\vartheta_{\min},\vartheta_1]$ if $a_1<b_1$. So for the second step of the bisection procedure we take
\begin{eqnarray*}
\begin{aligned}
\vartheta_2=
\begin{cases}
\frac{1}{2}(\vartheta_1+\vartheta_{\max}),& if~~~~ a_1>b_1,\\
\frac{1}{2}(\vartheta_1+\vartheta_{\min}),& if~~~~ a_1<b_1.
\end{cases}
\end{aligned}
\end{eqnarray*}
Let $$a_2=\max_{\phi_i-\varphi>\vartheta_2}f_i(\vartheta_2),~~ b_2=\max_{\phi_i-\varphi<\vartheta_2}f_{i}(\vartheta_2),~~ c_2=\max_{\phi_i-\varphi=\vartheta_2}f_i(\vartheta_2).$$
As before we can determine whether $\vartheta_2$ is a unique minimizer. If not, for $a_1>b_1$ any minimizer must be in $[\vartheta_2,\vartheta_{\max}]$ if $a_2>b_2$ and must be in $[\vartheta_1,\vartheta_{2}]$ if $a_2<b_2$; for $a_1<b_1$ any minimizer must be in $[\vartheta_2,\vartheta_{1}]$ if $a_2>b_2$ and must be in $[\vartheta_{\min},\vartheta_{2}]$
if $a_2<b_2$. We can then continue the bisection procedure. Unless a unique minimizer is found in a finite number of steps, we get a sequence $\{\vartheta_k\}$ with ${\rm lim}_{k\rightarrow \infty}\vartheta_k=\vartheta_*$ and $|\vartheta_k-\vartheta_*|\leq\frac{\vartheta_{\max}-\vartheta_{\min}}{2^k}<\frac{\pi}{2^k}$. By construction, $\vartheta_*$ is the only candidate for the minimizer. So it is the unique minimizer since the existence is already known.
\end{proof}

In step $k$ of the above bisection procedure, $\phi_i-\varphi,i=1,2,\cdots,m+n$ are divided into three parts: one with $\phi_i-\varphi>\vartheta_k$, one with $\phi_i-\varphi<\vartheta_k$, and the other with $\phi_i-\varphi=\vartheta_k$. We have assumed that this division is done in exact arithmetic. In practice this division is done by a computer and may be different from the division in exact arithmetic when some $\phi_i-\varphi$ are extremely close to $\vartheta_k$. But this will have very little effect on the accuracy of the computed $\vartheta_*$.

When $m=n$, our bisection procedure requires $O(n)$ operations each step, while the doubling algorithm requires $O(n^3)$ operations each step. We have already seen in the proof of Theorem \ref{thm16} that the $k$-th approximation $\vartheta_k$ to the minimizer $\vartheta_*$ satisfies $|\vartheta_k-\vartheta_*|<\frac{\pi}{2^k}$. So when $n$ is large, the computational work for using the bisection procedure to approximate $\vartheta_*$ to machine precision is negligible compared to the work for the doubling algorithm. In practice, there is no need to compute $\vartheta_*$ so accurately. If the $k$-th approximation is obtained by $\vartheta_k=\frac{1}{2}(\vartheta_a +\vartheta_b)$, we will stop the bisection if $|\vartheta_b-\vartheta_a|<tol$. We will take $tol=10^{-6}$ in our numerical experiments. The computational work is negligible.

A simple preprocessing procedure for the NARE (\ref{e1}) is then as follows: Use the bisection method
described in the proof of Theorem \ref{thm16} to determine a good approximation $\widetilde{\vartheta}$ to $\vartheta_*$, let $\chi=e^{-{\rm j}\widetilde{\vartheta}}$ and transform the NARE (\ref{e1}) to the NARE (\ref{e32}).

Let
$$\widetilde{\psi}_1=\max_{i=n+1,n+2,\cdots,m+n}\frac{p_{\omega,i}^\chi}{\varpi},~~~~\widetilde{\psi}_2=\max_{i=1,2,\cdots,n}\frac{p_{\omega,i}^\chi}{\varpi},$$
where $\frac{p_{\omega,i}^\chi}{\varpi},i=1,2,\cdots,m+n$ are as in (\ref{e33}). Then we can apply ADDA to (\ref{e32}) with $t>\widetilde{\psi}_1$ and $\gamma>\widetilde{\psi}_2$. We often have the situation that the two largest values of $f_i(\widetilde{\vartheta}˜)$, say $f_{i_1}(\widetilde{\vartheta}˜)$ and $f_{i_2}(\widetilde{\vartheta}˜)$ with $i_1<i_2$, are very close. If it happens for $i_1\in\{1,2,\cdots,n\}$ and $i_2\in\{n+1,n+2,\cdots,m+n\}$, then $\widetilde{\psi}_1\approx\widetilde{\psi}_2$ and we may take $t=\gamma$ for ADDA. In this case, ADDA is reduced to SDA.

While the solution sets for the NAREs (\ref{e1}) and (\ref{e32}) are the same, there are many solutions in the set. We still need to make sure that the required solutions $\Phi$ and $\Psi$ are obtained when ADDA is applied to the transformed equation (\ref{e32}).

\begin{theorem}\label{thm17}
Suppose the NARE (\ref{e1}) is in class $\mathbb{H}^\omega$ and $Q_\omega \textbf{1}> 0$. Let $\chi$ be any unimodular number such that $\omega {\rm Re}(\chi [Q]_{ii})+(1-\omega){\rm Im}(\chi [Q]_{ii})-q_i>0$ for all $i\in\{1,2,\cdots,m+n\}$ ($\chi=e^{-{\rm j}\widetilde{\vartheta}}$ in particular). Apply ADDA to the NARE (\ref{e1}) with $\alpha=tz_{\omega}^\bot$ and $\beta=\gamma z_{\omega}^\bot$. If
$$t\geq \widetilde{\psi}_1,~~~~\gamma\geq\widetilde{\psi}_2,$$
then the sequences $\{E_k\}$, $\{F_k\}$, $\{G_k\}$ and
$\{H_k\}$ are well-definied and $\{H_k\}$ and $\{G_k\}$ converge quadratically to $\Phi$ and $\Psi$, respectively.
\end{theorem}

ADDA stated by Theorem \ref{thm17} will be denoted by pADDA.

\begin{theorem}\label{thm28}
Suppose the NARE (\ref{e1}) is in class $\mathbb{H}^\omega$ and $Q_\omega \textbf{1}> 0$. Let $\chi$ be any unimodular number such that $\omega {\rm Re}(\chi [Q]_{ii})+(1-\omega){\rm Im}(\chi [Q]_{ii})-q_i>0$ for all $i\in\{1,2,\cdots,m+n\}$ ($\chi=e^{-{\rm j}\widetilde{\vartheta}}$ in particular). Apply SDA to the NARE (\ref{e1}) with $\alpha=tz_{\omega}^\bot$. If
$$t\geq\max\{\widetilde{\psi}_1,\widetilde{\psi}_2\},$$
then the sequences $\{E_k\}$, $\{F_k\}$, $\{G_k\}$ and
$\{H_k\}$ are well-definied and $\{H_k\}$ and $\{G_k\}$ converge quadratically to $\Phi$ and $\Psi$, respectively.
\end{theorem}

SDA stated by Theorem \ref{thm28} will be denoted by pSDA.

\subsubsection{Further choice of the parameters $t$ and $\gamma$}\

When $(\omega{\rm Im}([Q]_{ii})-(1-\omega){\rm Re}([Q]_{ii}))^2$ is large compared to $\omega {\rm Re}([Q]_{ii})+(1-\omega){\rm Im}([Q]_{ii})-q_i$ for some $i$, the numbers $\psi_1$ and/or $\psi_2$ will be large and the convergence of ADDA may be slow. To improve the performance of ADDA, we will find smaller parameters from the convergence region given by (\ref{e25}) and (\ref{e26}). The idea is to use the straight line $\gamma=ct$ to cut the convergence region, where the slope $c>0$ is to be chosen properly.

\begin{theorem}\label{thm18}
Suppose the NARE (\ref{e1}) is in class $\mathbb{H}^\omega$ and $Q_\omega \textbf{1}>0$. Let $\Phi$ and $\Psi$ be as in Theorem \ref{thm5}. Apply ADDA to the NARE (\ref{e1}) with $\alpha=tz_{\omega}^\bot$ and $\beta=\gamma z_{\omega}^\bot.$ Assume that $\gamma=ct$. If
\begin{eqnarray*}
\begin{aligned}
t>\max\{\eta_{\omega,1}(c),\eta_{\omega,2}(c)\},
\end{aligned}
\end{eqnarray*}
where
$$\eta_{\omega,1}(c)=\max_{i=1,2,\cdots,n}r_{\omega,i}(c),~~~~\eta_{\omega,2}(c)=\max_{i=n+1,n+2,\cdots,m+n}r_{\omega,i}(c)$$ with
\begin{eqnarray*}
\begin{aligned}
r_{\omega,i}(c)&=\frac{-(c-1)p_{\omega,i}+\sqrt{(c-1)^2p_{\omega,i}^2+4c(p_{\omega,i}^2-s_{\omega,i}^2)}}{2c\varpi},&i&=1,2,\cdots,n,\\
r_{\omega,i}(c)&=\frac{(c-1)p_{\omega,i}+\sqrt{(c-1)^2p_{\omega,i}^2+4c(p_{\omega,i}^2-s_{\omega,i}^2)}}{2c\varpi},&i&=n+1,n+2,\cdots,m+n,\\
\varpi&=\omega^2+(1-\omega)^2,&&
\end{aligned}
\end{eqnarray*}
then the sequences $\{E_k\}$, $\{F_k\}$, $\{H_k\}$, $\{G_k\}$ are well-defined and $\{H_k\}$ and $\{G_k\}$ converge quadratically to $\Phi$ and $\Psi$, respectively.

\end{theorem}

\begin{proof}
Since $\gamma=ct$, the inequalities (\ref{e30}) and (\ref{e31})
are equivalent to
\begin{eqnarray*}
\begin{aligned}
ct^2+(c-1)\frac{p_{\omega,i}}{\varpi}t-\frac{p_{\omega,i}^2}{\varpi^2}+\frac{s_{\omega,i}^2}{\varpi^2}&>0,&t&>\frac{q_i}{c\varpi},&i&=1,2,\cdots,n,\\
ct^2-(c-1)\frac{p_{\omega,i}}{\varpi}t-\frac{p_{\omega,i}^2}{\varpi^2}+\frac{s_{\omega,i}^2}{\varpi^2}&>0,&t&>\frac{q_i}{\varpi},&i&=n+1,n+2,\cdots,m+n.
\end{aligned}
\end{eqnarray*}
Let $r_{\omega,i}(c)$ be the larger root of  $ct^2+(c-1)\frac{p_{\omega,i}}{\varpi}t-\frac{p_{\omega,i}^2}{\varpi^2}+\frac{s_{\omega,i}^2}{\varpi^2}=0,i=1,2,\cdots,n$,
i.e.,
\begin{eqnarray*}
\begin{aligned}
r_{\omega,i}(c)&=\frac{-(c-1)\frac{p_{\omega,i}}{\varpi}+\sqrt{(c-1)^2\frac{p_{\omega,i}^2}{\varpi^2}+4c(\frac{p_{\omega,i}^2}{\varpi^2}-\frac{s_{\omega,i}^2}{\varpi^2})}}{2c}\\
&=\frac{-(c-1)p_{\omega,i}+\sqrt{(c-1)^2p_{\omega,i}^2+4c(p_{\omega,i}^2-s_{\omega,i}^2)}}{2c\varpi},~~~~i=1,2,\cdots,n.\\
\end{aligned}
\end{eqnarray*}
If $t>r_{\omega,i}(c)$, then $ct^2+(c-1)\frac{p_{\omega,i}}{\varpi}t-\frac{p_{\omega,i}^2}{\varpi^2}+\frac{s_{\omega,i}^2}{\varpi^2}>0$. Besides, we can prove $r_{\omega,i}(c)\geq\frac{q_i}{c\varpi},i=1,2,\cdots,n$. In fact, we have
\begin{eqnarray*}
(c-1)^2p_{\omega,i}^2+4c(p_{\omega,i}^2-s_{\omega,i}^2)-((c-1)p_{\omega,i}+2q_i)^2=4(c+1)q_is_{\omega,i}\geq0,
\end{eqnarray*}
so
\begin{eqnarray*}
r_{\omega,i}(c)-\frac{q_i}{c\varpi}=\frac{\sqrt{(c-1)^2p_{\omega,i}^2+4c(p_{\omega,i}^2-s_{\omega,i}^2)}-((c-1)p_{\omega,i}+2q_i)}{2c\varpi}\geq0.
\end{eqnarray*}
Let $r_{\omega,i}(c)$ be the larger root of  $ct^2-(c-1)\frac{p_{\omega,i}}{\varpi}t-\frac{p_{\omega,i}^2}{\varpi^2}+\frac{s_{\omega,i}^2}{\varpi^2}=0,i=n+1,n+2,\cdots,m+n$, i.e.,
\begin{eqnarray*}
\begin{aligned}
r_{\omega,i}(c)&=\frac{(c-1)\frac{p_{\omega,i}}{\varpi}+\sqrt{(c-1)^2\frac{p_{\omega,i}^2}{\varpi^2}+4c(\frac{p_{\omega,i}^2}{\varpi^2}-\frac{s_{\omega,i}^2}{\varpi^2})}}{2c}\\
&=\frac{(c-1)p_{\omega,i}+\sqrt{(c-1)^2p_{\omega,i}^2+4c(p_{\omega,i}^2-s_{\omega,i}^2)}}{2c\varpi},~~~~i=n+1,n+2,\cdots,m+n.
\end{aligned}
\end{eqnarray*}
Similarly, if $t>r_{\omega,i}(c)$, then $ct^2-(c-1)\frac{p_{\omega,i}}{\varpi}t-\frac{p_{\omega,i}^2}{\varpi^2}+\frac{s_{\omega,i}^2}{\varpi^2}>0$. We can also prove $r_{\omega,i}(c)\geq\frac{q_i}{\varpi},i=n+1,n+2,\cdots,m+n$. In fact, we have
\begin{eqnarray*}
(c-1)^2p_{\omega,i}^2+4c(p_{\omega,i}^2-s_{\omega,i}^2)-(2cq_i-(c-1)p_{\omega,i})^2=4c(c+1)q_is_{\omega,i}\geq0,
\end{eqnarray*}
so
\begin{eqnarray*}
r_{\omega,i}(c)-\frac{q_i}{\varpi}=\frac{\sqrt{(c-1)^2p_{\omega,i}^2+4c(p_{\omega,i}^2-s_{\omega,i}^2)}-(2cq_i-(c-1)p_{\omega,i})}{2c\varpi}\geq0.
\end{eqnarray*}
Thus if $t>\max\{\eta_{\omega,1}(c),\eta_{\omega,2}(c)\}$, we have
\begin{eqnarray*}
\begin{aligned}
(t+\frac{p_{\omega,i}}{\varpi})(\gamma-\frac{p_{\omega,i}}{\varpi})&>-\frac{s_{\omega,i}^2}{\varpi^2},&\gamma&>\frac{q_i}{\varpi},& i&=1,2,\cdots,n,\\
(\gamma+\frac{p_{\omega,i}}{\varpi})(t-\frac{p_{\omega,i}}{\varpi})&>-\frac{s_{\omega,i}^2}{\varpi^2},&t&>\frac{q_i}{\varpi},&i&=n+1,n+2,\cdots,m+n,
\end{aligned}
\end{eqnarray*}
then (\ref{e25}) and (\ref{e26}) are satisfied, then the sequences $\{E_k\}$, $\{F_k\}$, $\{H_k\}$, $\{G_k\}$ are well-defined and $\{H_k\}$ and $\{G_k\}$ converge quadratically to $\Phi$ and $\Psi$, respectively.
\end{proof}

Next, we consider how to apply SDA to the NARE (\ref{e1}). By taking $c=1$, we have
\begin{eqnarray*}
r_{\omega,i}(1)=\frac{\sqrt{q_i(\omega {\rm Re}([Q]_{ii})+(1-\omega){\rm Im}([Q]_{ii})+\frac{(\omega {\rm Im}([Q]_{ii})-(1-\omega){\rm Re}([Q]_{ii}))^2}{\omega {\rm Re}([Q]_{ii})+(1-\omega){\rm Im}([Q]_{ii})-q_i})}}{\varpi},~~~~i=1,2,\cdots,m+n.
\end{eqnarray*}

\begin{theorem}\label{thm19}
Suppose the NARE (\ref{e1}) is in class $\mathbb{H}^\omega$ and $Q_\omega \textbf{1}>0$. Let $\Phi$ and $\Psi$ be as in Theorem \ref{thm5}. Apply SDA to the NARE (\ref{e1}) with $\alpha=tz_{\omega}^\bot$. If
$$t>\max_{i=1,2,\cdots,m+n}\tau_{\omega,i}$$
where
\begin{eqnarray*}
\begin{aligned}
\tau_{\omega,i}&=\frac{\sqrt{q_i(\omega {\rm Re}([Q]_{ii})+(1-\omega){\rm Im}([Q]_{ii})+\frac{(\omega {\rm Im}([Q]_{ii})(1-\omega){\rm Re}([Q]_{ii}))^2}{\omega {\rm Re}([Q]_{ii})+(1-\omega){\rm Im}([Q]_{ii})-q_i})}}{\varpi},~~~~i=1,2,\cdots,m+n,\\
\varpi&=\omega^2+(1-\omega)^2,
\end{aligned}
\end{eqnarray*}
then the sequences $\{E_k\},\{F_k\},\{H_k\},\{G_k\}$ are well-defined and $\{H_k\}$ and $\{G_k\}$ converge quadratically to $\Phi$ and $\Psi$, respectively.
\end{theorem}

Notice that $\max_{i=1,2,\cdots,m+n}\tau_{\omega,i}<\max_{i=1,2,\cdots,m+n}{\frac{p_{\omega,i}}{\varpi}}$. So we now allow smaller values of the parameter $t$ for SDA. By a more careful choice of $c$ in Theorem \ref{thm18}, we can allow smaller values of the parameters $t$ and $\gamma$ for ADDA.

\begin{theorem}\label{thm20}
Suppose the NARE (\ref{e1}) is in class $\mathbb{H}^\omega$ and $Q_\omega \textbf{1}>0$. Let $\Phi$ and $\Psi$ be as in Theorem \ref{thm5}. Then there is a unique $c^*>0$ such that $\max_{i=1,2,\cdots,n}r_{\omega,i}(c^*)=\max_{i=n+1,n+2,\cdots,m+n}r_{\omega,i}(c^*)$. Let $t^*=\max_{i=n+1,n+2,\cdots,m+n}r_{\omega,i}(c^*)$ and $\gamma^*=c^*t^*$. Then for any $t$ and $\gamma$ satisfying (\ref{e30}) and (\ref{e31}), we have $t>t^*$ and $\gamma>\gamma^*$. In particular, $\psi_1>t^*$, $\psi_2>\gamma^*$.
\end{theorem}

\begin{proof}

For $i=1,2,\cdots,n$, the derivative of $r_{\omega,i}(c)$ about $c$ is
$$r_{\omega,i}(c)^{'}=\frac{-(c+1)p_{\omega,i}^2+2cs_{\omega,i}^2-p_{\omega,i}\sqrt{(c-1)^2p_{\omega,i}^2+4c(p_{\omega,i}^2-s_{\omega,i}^2)}}{2c^2\varpi\sqrt{(c-1)^2p_{\omega,i}^2+4c(p_{\omega,i}^2-s_{\omega,i}^2)}}<0$$
since
$$(2cs_{\omega,i}^2-(c+1)p_{\omega,i}^2)^2-p_{\omega,i}^2((c-1)^2p_{\omega,i}^2+4c(p_{\omega,i}^2-s_{\omega,i}^2))=4c^2s_{\omega,i}^2(s_{\omega,i}^2-p_{\omega,i}^2)<0.$$
So $r_{\omega,i}(c)$ ($i=1,2,\cdots,n$) is strictly decreasing on $c\in(0,\infty)$, from $\infty$ to $0$ for each $\omega\in[0,1]$. It follows that $\eta_{\omega,1}(c)$ is strictly decreasing on $(0,\infty)$, from $\infty$ to $0$ for each $\omega\in[0,1]$.

Similarly, for $i=n+1,n+2,\cdots,m+n$, the derivative of $r_{\omega,i}(c)$ about $c$ is
$$r_{\omega,i}(c)^{'}=\frac{-(c+1)p_{\omega,i}^2+2cs_{\omega,i}^2+p_{\omega,i}\sqrt{(c-1)^2p_{\omega,i}^2+4c(p_{\omega,i}^2-s_{\omega,i}^2)}}{2c^2\varpi\sqrt{(c-1)^2p_{\omega,i}^2+4c(p_{\omega,i}^2-s_{\omega,i}^2)}}>0$$
since
$$p_{\omega,i}^2((c-1)^2p_{\omega,i}^2+4c(p_{\omega,i}^2-s_{\omega,i}^2))-((c+1)p_{\omega,i}^2-2cs_{\omega,i}^2)^2=4c^2s_{\omega,i}^2(p_{\omega,i}^2-s_{\omega,i}^2)>0.$$
So $r_{\omega,i}(c)$ ($i=n+1,n+2,\cdots,m+n$) is strictly increasing on $(0,\infty)$, from $\frac{p_{\omega,i}^2-s_{\omega,i}^2}{\varpi p_{\omega,i}}$ to $\frac{p_{\omega,i}}{\varpi}$ for each $\omega\in[0,1]$. It follows that $\eta_{\omega,2}(c)$ is strictly increasing on $(0,\infty)$, from $\max_{i=n+1,n+2,\cdots,m+n}\frac{p_{\omega,i}^2-s_{\omega,i}^2}{\varpi p_{\omega,i}}$ to $\max_{i=n+1,n+2,\cdots,m+n}\frac{p_{\omega,i}}{\varpi}$  for each $\omega\in[0,1]$.

We only need to show that $t>t^*$ and $\gamma>\gamma^*$ for any $t$ and $\gamma$ satisfying (\ref{e30}) and (\ref{e31}). In fact, $\gamma=ct$ with $c=\gamma/t$ and the conditions (\ref{e30}) and (\ref{e31}) require that
$$t>\max\{\max_{i=1,2,\cdots,n}r_{\omega,i}(c),\max_{i=n+1,n+2,\cdots,m+n}r_{\omega,i}(c)\}\geq\max\{\max_{i=1,2,\cdots,n}r_{\omega,i}(c^*),\max_{i=n+1,n+2,\cdots,m+n}r_{\omega,i}(c^*)\}=t^*.$$
For any $c>0$ we replace $t$ by $\gamma/c$ in (\ref{e30}) and (\ref{e31}), and find as before that (\ref{e30}) is equivalent to $\gamma>c\eta_{\omega,1}(c)$ and (\ref{e31}) is equivalent to $\gamma>c\eta_{\omega,2}(c)$. So we need $\gamma>\max\{c\eta_{\omega,1}(c),c\eta_{\omega,2}(c)\}$. As before we can show that $c\eta_{\omega,1}(c)$ is strictly decreasing on $(0,\infty)$, from $\max_{i=1,2,\cdots,n}\frac{p_{\omega,i}}{\varpi}$ to $\max_{i=1,2,\cdots,n}\frac{p_{\omega,i}^2-s_{\omega,i}^2}{\varpi p_{\omega,i}}$ for each $\omega\in[0,1]$,  and that $c\eta_{\omega,2}(c)$ is strictly increasing on $(0,\infty)$, from $0$ to $\infty$ for each $\omega\in[0,1]$. It follows that $\gamma>\max\{c\eta_{\omega,1}(c),c\eta_{\omega,2}(c)\}\geq\max\{c^*\eta_{\omega,1}(c^*),c^*\eta_{\omega,2}(c^*)\}=\gamma^*$.

\end{proof}

From the above proof, we can also see that
$$t^*>\underline{t}=\max_{i=n+1,n+2,\cdots,m+n}\frac{p_{\omega,i}^2-s_{\omega,i}^2}{\varpi p_{\omega,i}},$$
$$\gamma^*>\underline{\gamma}=\max_{i=1,2,\cdots,n}\frac{p_{\omega,i}^2-s_{\omega,i}^2}{\varpi p_{\omega,i}}.$$
It follows that $$\underline{\gamma}/\max_{i=n+1,n+2,\cdots,m+n}\frac{p_{\omega,i}}{\varpi}<c^*<\max_{i=1,2,\cdots,n}\frac{p_{\omega,i}}{\varpi}/\underline{t},$$
so $c^*$ can be found by the usual bisection method applied to the function $\eta_{\omega,1}(c)-\eta_{\omega,2}(c)$ on the interval $$[\underline{\gamma}/\max_{i=n+1,n+2,\cdots,m+n}\frac{p_{\omega,i}}{\varpi},\max_{i=1,2,\cdots,n}\frac{p_{\omega,i}}{\varpi}/\underline{t}].$$

While Theorem \ref{thm20} allows us to use smaller parameters $t$ and $\gamma$ for ADDA, the smaller parameters will not always provide better convergence. Generally speaking, the smaller $\rho(\mathcal{R})\rho(\mathcal{S})$ is, the faster ADDA converges. In this regard, we are to choose parameters $t$ and $\gamma$ to make $\rho(\mathcal{R})\rho(\mathcal{S})$ as small as possible. Once again we fix $c>0$ and let $\gamma=ct$. We will try to find good values for $t$.

Let $\lambda$ and $\mu$ be any eiegnvalues of $R$ and $S$, respectively. Note that $\lambda$ is an eigenvalue  of
$$H=\left(
\begin{matrix}
D&-C\\
B&-A
\end{matrix}
\right)$$
in the upper right of $L(z_\omega)$ and $-\mu$ is an eigenvalue  of $H$ in the lower left of $L(z_\omega)$. It follows from Gershgorin's theorem and triangle inequality that
$$|\lambda|\leq \max_{i=1,2,\cdots,n}(|[Q]_{ii}|+q_i),~~~~~|\mu|\leq \max_{i=n+1,n+2,\cdots,m+n}(|[Q]_{ii}|+q_i).$$
The eigenvalue of $\mathcal{R}$ and $\mathcal{S}$ are $\frac{\lambda-c\alpha}{\lambda+\alpha}=\frac{\lambda-ctz_{\omega}^\bot}{\lambda+tz_{\omega}^\bot}$ and $\frac{\mu-\alpha}{\mu+c\alpha}=\frac{\mu-tz_{\omega}^\bot}{\mu+ctz_{\omega}^\bot}$.

\begin{proposition}\label{thm21}
Let $c>0$ and $x=a+{\rm j}b$ with $x$ being in the upper right of $L(z_\omega)$. Then the function
$$f_\omega(t)=|\frac{x-ctz_{\omega}^\bot}{x+tz_{\omega}^\bot}|$$
is increasing when
\begin{eqnarray}\label{e34}
t\geq\frac{-(c-1)\varpi|x|^2+\sqrt{(c-1)^2\varpi^2|x|^4+4c\varpi\theta^2 |x|^2}}{2c\varpi\theta},
\end{eqnarray}
where $\varpi=\omega^2+(1-\omega)^2$ and $\theta=\omega a+(1-\omega)b$.
\end{proposition}

\begin{proof}
We have
\begin{eqnarray*}
f_\omega(t)=\sqrt{\frac{(a-ct\omega)^2+(b-ct(1-\omega))^2}{(a+t\omega)^2+(b+t(1-\omega))^2}},
\end{eqnarray*}
a simple computation shows that $f_\omega(t)^{\prime}\geq0$
if and only if
\begin{eqnarray}\label{e35}
\begin{aligned}
&c(a\omega+b(1-\omega)-tc\varpi)(|x|^2+t^2\varpi+2at\omega+2bt(1-\omega))\\
&+(|x|^2+c^2t^2\varpi-2act\omega-2bct(1-\omega))(a\omega+b(1-\omega)+t\varpi)\leq0.
\end{aligned}
\end{eqnarray}
Let $\theta=\omega a+(1-\omega)b$, then (\ref{e35}) is equivalent to
\begin{eqnarray*}
c(\theta-tc\varpi)(|x|^2+t^2\varpi+2t\theta)+(|x|^2+c^2t^2\varpi-2ct\theta)(\theta+t\varpi)\leq0,
\end{eqnarray*}
equivalently,
$$t^2c^2\varpi\theta+t^2c\varpi\theta+|x|^2t\varpi(c^2-1)-|x|^2\theta(c+1)\geq0,$$
i.e.,
$$ c\varpi\theta t^2+(c-1)\varpi|x|^2t-|x|^2\theta\geq0.$$
Because the function $c\varpi\theta t^2+(c-1)\varpi|x|^2t-|x|^2\theta$ about $t$ is quadratic and its discriminant $\Omega=(c-1)^2\varpi^2|x|^4+4c\varpi\theta^2 |x|^2\geq 0$ and $\theta=a\omega+b(1-\omega)>0$. Thus we have that if
$$t\geq\frac{-(c-1)\varpi|x|^2+\sqrt{(c-1)^2\varpi^2|x|^4+4c\varpi\theta^2 |x|^2}}{2c\varpi\theta},$$
then $f_\omega(t)$ is increasing.

\end{proof}

If we use SDA, then $c=1$ and (\ref{e34}) simplifies to $$t\geq\frac{|x|\sqrt{\varpi}}{\varpi}=\frac{|x|}{\sqrt{\varpi}}.$$
It follows that $\rho(\mathcal{R})\rho(\mathcal{S})$ is increasing in $t$ if
\begin{eqnarray}\label{e36}
t\geq \frac{\max_{i=1,2,\cdots,m+n}(|[Q]_{ii}|+q_i)}{\sqrt{\varpi}}.
\end{eqnarray}

Note however that this is only a sufficient condition and $\rho(\mathcal{R})\rho(\mathcal{S})$ may be smaller for some smaller $t$ values. In this regard, when $\frac{\max_{i=1,2,\cdots,m+n}(|[Q]_{ii}|+q_i)}{\sqrt{\varpi}}\geq\max\{\psi_1,\psi_2\}=\psi^*$, we will take $t=\psi^*=\max\{\psi_1,\psi_2\}$, i.e., we stick to the original strategy for choosing $t$ for SDA. Now suppose that $\frac{\max_{i=1,2,\cdots,m+n}(|[Q]_{ii}|+q_i)}{\sqrt{\varpi}}<\psi^*=\max\{\psi_1,\psi_2\}$. Since strict inequality is required in Theorem \ref{thm19} we will require $t\geq \tau_{\omega}^*=\zeta \max_{i=1,2,\cdots,m+n}\tau _{\omega,i}$, where $\zeta$ is slightly bigger than 1 (we take $\zeta=1.01$ in our numerical experiments). Our strategy for choosing $t$ in SDA is then to take

$$t=\max\{\tau_{\omega}^*,\frac{1}{2}\frac{\max_{i=1,2,\cdots,m+n}(|[Q]_{ii}|+q_i)}{\sqrt{\varpi}}\},$$
where the factor $\frac{1}{2}$ is introduced to account for the fact that (\ref{e36}) is only a sufficient condition. For the NARE (\ref{e1}), SDA with this new parameter strategy will be denoted by SDAn.

The situation for ADDA is more complicated since the inequality in (\ref{e34}) is complicated when $c\neq1$. When $c\geq1$ it is easy to see that (\ref{e34}) holds if $t\geq \frac{|x|}{\sqrt{c\varpi}}$. But no such simplification is available when $c<1$. When we apply Proposition \ref{thm21} to $\frac{\lambda-c\alpha}{\lambda+\alpha}$ and $\frac{\mu-\alpha}{\mu+c\alpha}=\frac{\mu-\frac{1}{c}c\alpha}{\mu+c\alpha}$, we will run into difficulties when $c\neq1$ since either $c$ or $\frac{1}{c}$ will be smaller than 1. Since we have no useful monotonicity results to apply for ADDA, our parameter strategy is solely based on Theorem \ref{thm20}. We compute $c^*$ by bisection method and take $t=\zeta t^*=\zeta\eta_{\omega,1}(c^*)$ and $\gamma=c^*t^*$, where $\zeta$ is slightly bigger than $1$ (we take $\zeta=1.01$ in our numerical experiments). For the NARE (\ref{e1}), ADDA with this new parameter strategy will be denoted by ADDAn.

Since there is more uncertainty about ADDAn, it may be appropriate to use SDAn when $0.1<\frac{\psi_1}{\psi_2}<10$ and use ADDAn otherwise. For the NARE (\ref{e1}), this method will be denoted by DAn. Since the bounds $0.1$ and $10$ are somewhat arbitrary, one cannot expect DAn to be always better than SDAn and ADDAn.

Note that our new parameter strategies are also applied to the NARE (\ref{e32}) which can be transformed from the NARE (\ref{e1}) by a precessing procedure. For the NARE (\ref{e32}), new parameter strategies will be denoted by pSDAn, pADDAn and pDAn, respectively.

\begin{remark}\label{thm30}

1. When $\omega=0$, $L(z_\omega,0)$ will be the real axis. The parameters $\alpha$ and $\beta$ can be found in the ray $R(z_\omega^\bot)$, i.e., the upper half imaginary axis.

2. When $\omega=1$, $L(z_\omega,0)$ will be the imaginary axis. The parameters $\alpha$ and $\beta$ can be found in the ray $R(z_\omega^\bot)$, i.e., the right half real axis.

3. When $\omega$ ranges from $0$ to $1$, the straight line $L(z_\omega,0)$ rotates clockwise from the real axis to the imaginary axis at the center of origin. The parameters $\alpha$ and $\beta$ can be found in the ray $R(z_\omega^\bot)$, it is a ray located in the first quadrant.

In a word, all of our results generalize the results obtained by \cite{Axe5}.

\end{remark}

\section{Numerical experiments}

In this section, we present some numerical examples to illustrate the effectiveness of our methods, including Newton's method, the fixed-point iterative methods and the two doubling algorithms:ADDA and SDA. Besides, the effectiveness of our preprocessing technique and new parameter selection strategies for ADDA and SDA is also demonstrated. All experiments are performed under Windows 7 and MATLAB(R2014a) running on a Lenovo desktop with an Intel(R) Core(TM) i5-4590, CPU at 3.30 GHz and 4 GB of memory. An algorithm for computing the extremal solution of the NARE (\ref{e1}) is terminated when the approximate solution $\Phi_k$ satisfies ${\rm NRes}<10^{-12}$, where
\begin{eqnarray*}
{\rm NRes}=\frac{\|\Phi_k C\Phi_k-\Phi_k D-A\Phi_k+B\|_1}{\|\Phi_k\|_1(\|\Phi_k\|_1\|C\|_1+\|D\|_1+\|A\|_1)+\|B\|_1}
\end{eqnarray*}
is the normalized residual. We use IT and CPU to denote the numbers and the consumed time of iterations, respectively.

For demonstrating the convergence of the fixed-point iterative methods, we will take TFP as a representative. The convergence of JFP, GSFP, SORFP and AORFP can be demonstrated easily. We apply ADDA and SDA to the NARE (\ref{e1}) directly, and for ADDA we take $t=\psi_1$ and $\gamma=\psi_2$, and for SDA we take $t=\max\{\psi_1,\psi_2\}$. ADDA and SDA with our preprocessing procedure will be denoted by pADDA and pSDA, respectively. We apply pADDA  and pSDA to the NARE (\ref{e32}) with $\chi=e^{-{\rm j} \vartheta_*}$. For pADDA we take $t=\widetilde{\psi}_1$ and $\gamma=\widetilde{\psi}_2$, and for pSDA we take $t=\max\{\widetilde{\psi}_1,\widetilde{\psi}_2\}$.

Example 7.1 below demonstrates the performance of our proposed methods for the NARE (\ref{e1}) in the situation where the matrix $Q$ may have same diagonal elements. In this situation, ADDA and pADDA are reduced to SDA and pSDA, respectively. Besides, a common value $\chi=e^{-{\rm j} \vartheta}=e^{-{\rm j}(\phi_i-\varphi)}$ can be found such that $f_i(\vartheta)$ attains the minimum for all $i$. And thus the bisection procedure in preprocessing procedure is unnecessary.

\textbf{Example 7.1.} Let $A,B,C,D\in \mathbb{C}^{n\times n}$ be given by

$$A=P+({\rm j}\eta)I_n,~~~~ D=P+( {\rm j}\eta)I_n,~~~~B=C=u I_n,$$
where $u\in (0,2),\eta\in \mathbb{R}$  and
$$P=\left(
\begin{matrix}
\xi&-1\\
&\xi&-1\\
&&\ddots&\ddots\\
&&&\xi&-1\\
-1&&&&\xi
\end{matrix}
\right).$$

\begin{table}[!htbp]
\caption{The numerical results for Example 7.1 with $n=512$.}
\centering
\footnotesize{
\setlength{\tabcolsep}{1.3mm}
   \begin{tabular}{cc|c|cccccccccc}
   \hline
   &&&\multicolumn{10}{c}{$u=0.01$}\\
   \hline
   $\omega$&$(\xi,\eta)$&&Newton&TFP&SDA&ADDAn&SDAn&DAn&pSDA&pADDAn&pSDAn&pDAn\\
   \hline
   \multirow{2}{*}{0}&\multirow{2}{*}{(-5,1.05)}&IT&2&3&16&12&12&12&4&5&4&4\\
   &&CPU&29.6298&37.9619&10.9125&9.4652&9.4989&9.9400&6.4912&8.9096&7.0260&6.4569\\
   \hline
   \multirow{2}{*}{0.1}&\multirow{2}{*}{(-3,1.5)}&IT&2&3&13&10&10&10&4&4&4&4\\
   &&CPU&31.7026&38.9788&10.0546&9.5092&9.5827&9.5305&9.3744&9.5600&9.2472&9.3298\\
   \hline
   \multirow{2}{*}{0.5}&\multirow{2}{*}{(-1,4)}&IT&2&3&7&6&6&6&4&4&4&4\\
   &&CPU&29.7952&38.0195&10.8399&10.5415&10.6214&10.6065&8.1645&8.6008&8.1516&8.1328\\
   \hline
   \multirow{2}{*}{0.9}&\multirow{2}{*}{(0,11)}&IT&2&2&15&11&11&11&4&5&4&4\\
   &&CPU&29.6355&26.0940&10.9842&9.6224&10.0095&9.6687&4.1576&5.3052&4.1425&4.1373\\
   \hline
   \multirow{2}{*}{1}&\multirow{2}{*}{(1.05,-5)}&IT&2&3&16&12&12&12&4&5&4&4\\
   &&CPU&29.7714&38.9348&10.9525&9.4719&9.4764&9.5643&6.7502&8.8732&6.4463&6.4653\\
   \hline
   \end{tabular}
   \begin{tabular}{cc|c|cccccccccc}
   \hline
   &&&\multicolumn{10}{c}{$u=0.1$}\\
   \hline
   $\omega$&$(\xi,\eta)$&&Newton&TFP&SDA&ADDAn&SDAn&DAn&pSDA&pADDAn&pSDAn&pDAn\\
   \hline
   \multirow{2}{*}{0}&\multirow{2}{*}{(-10,1.2)}&IT&2&3&15&11&11&11&4&5&4&4\\
   &&CPU&29.9453&37.1055&10.6396&9.9290&9.9784&9.9687&4.2959&5.4316&4.2641&4.3756\\
   \hline
   \multirow{2}{*}{0.1}&\multirow{2}{*}{(-5,2)}&IT&3&4&10&8&8&8&4&4&4&4\\
   &&CPU&43.3523&50.4984&9.3162&8.8982&8.9718&8.9615&6.6374&7.0228&6.6194&6.6003\\
   \hline
   \multirow{2}{*}{0.5}&\multirow{2}{*}{(-1,5)}&IT&3&4&6&6&6&6&4&4&4&4\\
   &&CPU&45.2348&51.8292&10.3645&10.9073&10.9314&10.9917&6.8078&7.1063&6.7415&6.7997\\
   \hline
   \multirow{2}{*}{0.9}&\multirow{2}{*}{(0,12)}&IT&2&3&14&10&10&10&4&5&4&4\\
   &&CPU&29.9612&44.5258&10.9930&9.7872&9.5895&9.7139&4.1036&5.2070&4.0702&4.1008\\
   \hline
   \multirow{2}{*}{1}&\multirow{2}{*}{(1.2,-10)}&IT&2&3&15&11&11&11&4&5&4&4\\
   &&CPU&30.3147&37.6761&10.7674&9.9678&10.0018&10.0234&4.2497&5.4287&4.2365&4.2334\\
   \hline
   \end{tabular}
   \begin{tabular}{cc|c|cccccccccc}
   \hline
   &&&\multicolumn{10}{c}{$u=1$}\\
   \hline
   $\omega$&$(\xi,\eta)$&&Newton&TFP&SDA&ADDAn&SDAn&DAn&pSDA&pADDAn&pSDAn&pDAn\\
   \hline
   \multirow{2}{*}{0}&\multirow{2}{*}{(-50,2.01)}&IT&2&4&20&13&13&13&4&6&4&4\\
   &&CPU&29.7309&50.4191&14.9592&12.5041&12.4715&12.4394&2.6849&3.8142&2.6675&2.6483\\
   \hline
   \multirow{2}{*}{0.1}&\multirow{2}{*}{(-10,5)}&IT&3&5&7&6&6&6&4&4&4&4\\
   &&CPU&44.7391&64.0577&9.8632&9.5124&9.5572&9.5649&4.3355&4.4114&4.2914&4.3118\\
   \hline
   \multirow{2}{*}{0.5}&\multirow{2}{*}{(-2,7)}&IT&3&7&7&6&6&6&4&4&4&4\\
   &&CPU&44.0267&89.8634&11.0217&10.9717&10.9457&10.9726&5.6055&5.4055&5.3099&5.3105\\
   \hline
   \multirow{2}{*}{0.9}&\multirow{2}{*}{(0,21)}&IT&3&4&14&9&9&9&4&5&4&4\\
   &&CPU&44.9850&50.5941&13.5404&10.8258&10.8532&10.8801&3.4014&4.2805&3.3786&3.3771\\
   \hline
   \multirow{2}{*}{1}&\multirow{2}{*}{(2.01,-50)}&IT&2&4&20&13&13&13&4&6&4&4\\
   &&CPU&30.0403&50.6845&14.9062&12.4521&12.4635&12.5044&2.6612&3.8160&2.6324&2.6395\\
   \hline
   \end{tabular}
   }
\end{table}

For Example 7.1, we test the performance of our methods for different $\omega$, $\xi$, $\eta$ and $u$. We let the size $n=512$. The numerical results are shown in Table 7.1. We can see that Newton's method converges fastest among all proposed methods, but it consumes too much time. Besides, we can also observe that FTP converges faster than SDA, but FTP consumes more time than SDA. Among all doubling algorithms, pSDA converges faster than SDA; ADDAn, SDAn and DAn converge faster than SDA; pADDAn, pSDAn and pDAn convege faster than SDA, ADDAn, SDAn and DAn. DAn chooses the faster method between ADDAn and SDAn, and pDAn chooses the faster method between pADDAn and pSDAn for this example. We must notice that pADDAn may converge slower than pADDA. This observation may be explained from the fact that smaller parameters may not result in faster algorithms for doubling algorithms. We observe that pSDA, pADDA, pSDAn and pDAn converge fastest among all doubling algorithms for this example.

Example 7.2 below demonstrates the performance of our proposed methods for the NARE (\ref{e1}) in the situation where the matrix $Q$ may have different diagonal elements. In this situation, we can't find a common value $\chi=e^{-{\rm j} \vartheta}$ such that $f_i(\vartheta)$ attains the minimum for all $i$, and we need take some time to perform the bisection procedure in the preprocessing procedure.

\textbf{Example 7.2.}  For $n=512,$ let

$$A=\eta\left(
\begin{matrix}
I_{n/2}&\\
&-I_{n/2}
\end{matrix}
\right)I+3*{\rm j},~D=2\eta\left(
\begin{matrix}
I_{n/2}&\\
&-I_{n/2}
\end{matrix}
\right)+3*{\rm j},~B=I_{n},~
C=I_{n}.$$

\begin{table}[!htbp]
\caption{The numerical results for Example 7.2 with $n=512$.}
\centering
\footnotesize{
\setlength{\tabcolsep}{1mm}
   \begin{tabular}{c|c|cccccccccccc}
   \hline
   &&&\multicolumn{10}{c}{$\omega=0$}\\
   \hline
 $\eta$&&Newton&TFP&ADDA&SDA&ADDAn&SDAn&DAn&pADDA&pSDA&pADDAn&pSDAn&pDAn\\
   \hline
   \multirow{2}{*}{-20}&IT&3&4&8&10&7&7&7&8&10&7&7&7\\
   &CPU&58.4497&73.0989&2.3950&2.8773&2.0640&2.0422&2.0788&2.5258&3.0417&2.2458&2.1973&2.2419\\
   \hline
   \multirow{2}{*}{-10}&IT&3&5&7&8&6&6&6&7&8&6&6&6\\
   &CPU&56.6389&84.5592&2.1478&2.3578&1.8240&1.8499& 1.8526&2.1427&2.4126&1.8741&1.9059&1.8623\\
   \hline
   \multirow{2}{*}{-5}&IT&3&6&5&7&5&6&6&5&7&5&6&6\\
   &CPU&57.6388&98.7460&1.6236&2.1024&1.5899&1.8365&1.8188&1.7059&2.2481&1.7125&2.0061&1.9677\\
   \hline
   \multirow{2}{*}{0}&IT&4&10&4&4&4&4&4&4&4&4&4&4\\
   &CPU&69.9680&145.0492&0.8615&0.8934&0.8588&0.8539&0.8622&0.8428&0.8957&0.8889&0.8864&0.9051\\
   \hline
   \multirow{2}{*}{5}&IT&3&6&5&7&5&6&6&5&7&5&6&6\\
   &CPU&57.7678&96.9707&1.6153&2.0731&1.5593&1.8183&1.8211&1.7015&2.2595&1.6744&1.9892&1.9588\\
   \hline
   \multirow{2}{*}{10}&IT&3&5&7&8&6&6&6&7&8&6&6&6\\
   &CPU&57.8238&84.4355&2.1602&2.3556&1.8236&1.8634&1.8367&2.1626&2.4128&1.8675&1.9107&1.8655\\
   \hline
   \multirow{2}{*}{20}&IT&3&4&8&10&7&7&7&8&10&7&7&7\\
   &CPU&57.4388&71.1231&2.4200&2.8483&2.0967&2.0673&2.1081&2.5161&3.0977&2.2476&2.2522&2.2352\\
   \hline
   \end{tabular}
   \begin{tabular}{c|c|cccccccccccc}
   \hline
   &&&\multicolumn{10}{c}{$\omega=0.1$}\\
   \hline
 $\eta$&&Newton&TFP&ADDA&SDA&ADDAn&SDAn&DAn&pADDA&pSDA&pADDAn&pSDAn&pDAn\\
   \hline
   \multirow{2}{*}{-8}&IT&3&5&8&13&7&8&7&6&8&6&6&6\\
   &CPU&57.1384&85.7310&2.3625&3.5691&2.1001&2.3310&2.1114&1.9961&2.5491&1.9967&1.9912&1.9514\\
   \hline
   \multirow{2}{*}{-4}&IT&3&6&6&7&5&6&6&5&6&5&5&5\\
   &CPU&58.2886&94.0871&1.8640&2.0755&1.5566&1.8467&1.8080&1.6114&1.8708&1.6243&1.6179&1.5984\\
   \hline
   \multirow{2}{*}{-1}&IT&4&9&4&4&4&4&4&4&4&4&4&4\\
   &CPU&70.4656&133.0604&1.3229&1.3265&1.3301&1.3124&1.2978&1.3424&1.3510&1.3502&1.3595&1.3298\\
   \hline
   \multirow{2}{*}{0}&IT&4&10&4&4&4&4&4&4&4&4&4&4\\
   &CPU&70.5027&147.2159&1.3492&1.3400&1.3215&1.3341&1.3284&1.4362&1.4414&1.4462&1.4307&1.4174\\
   \hline
   \multirow{2}{*}{1}&IT&4&9&4&4&4&4&4&4&4&4&4&4\\
   &CPU&69.6891&133.1174&1.3567&1.3174&1.3221&1.3221&1.2696&1.3460&1.8313&1.3482&1.3476&1.3328\\
   \hline
   \multirow{2}{*}{4}&IT&3&6&6&7&5&6&6&5&6&5&5&5\\
   &CPU&57.0302&95.8556&1.8563&2.0993&1.5799&1.8120&1.8292&1.6144&1.8739&1.6128&1.6306&1.5934\\
   \hline
   \multirow{2}{*}{8}&IT&3&5&8&13&7&8&7&6&8&6&6&6\\
   &CPU&57.7937&84.3933&2.3935&3.5873&2.0995&2.3519&2.1214&1.9627&2.5571&1.9775&1.9913&1.9290\\
   \hline
   \end{tabular}
   \begin{tabular}{c|c|cccccccccccc}
   \hline
   &&&\multicolumn{10}{c}{$\omega=0.5$}\\
   \hline
   $\eta$&&Newton&TFP&ADDA&SDA&ADDAn&SDAn&DAn&pADDA&pSDA&pADDAn&pSDAn&pDAn\\
   \hline
   \multirow{2}{*}{-0.45}&IT&4&9&6&8&5&6&6&4&3&4&3&3\\
   &CPU&70.7762&133.8419&1.8694&2.3648&1.5755&1.8512&1.8348&1.4298&1.2643&1.4479&1.1447&1.1438\\
   \hline
   \multirow{2}{*}{-0.3}&IT&4&9&6&6&5&5&5&3&3&4&3&3\\
   &CPU&70.0761&134.2049&1.8255&1.8351&1.5679&1.5860&1.6046&1.1846&1.1592&1.5908&1.1601&1.1384\\
   \hline
   \multirow{2}{*}{-0.15}&IT&4&9&5&5&5&5&5&3&3&4&3&3\\
   &CPU&69.2073&133.8227&1.5746&1.5703&1.6081&1.5855&1.6146&1.1578&1.1762&1.4354&1.1664&1.1516\\
   \hline
   \multirow{2}{*}{0}&IT&4&10&5&5&4&4&4&3&3&4&3&3\\
   &CPU&69.7512&145.1118&1.6052&1.5625&1.3468&1.3138&1.3153&1.1592&1.3629&1.4478&1.1765&1.1472\\
   \hline
   \multirow{2}{*}{0.15}&IT&4&9&5&5&5&5&5&3&3&4&3&3\\
   &CPU&70.5676&135.1130&1.6306&1.5829&1.5304&1.6044&1.6004&1.1546& 1.1481&1.4354&1.1637&1.1447\\
   \hline
   \multirow{2}{*}{0.3}&IT&4&9&6&6&5&5&5&3&3&4&3&3\\
   &CPU&70.9305&131.3888&1.8863&1.7906&1.5782&1.5405&1.5636&1.1681&1.1545&1.4178&1.1809&1.1561\\
   \hline
   \multirow{2}{*}{0.45}&IT&4&9&6&8&5&6&6&4&3&4&3&3\\
   &CPU&71.2917&130.1491&1.8720&2.3765&1.5757&1.8181&1.8227&1.4440&1.0912&1.4321&1.1646&1.1388\\
   \hline
   \end{tabular}
   \begin{tabular}{c|c|cccccccccccc}
   \hline
   &&&\multicolumn{10}{c}{$\omega=0.9$}\\
   \hline
 $\eta$&&Newton&TFP&ADDA&SDA&ADDAn&SDAn&DAn&pADDA&pSDA&pADDAn&pSDAn&pDAn\\
   \hline
   \multirow{2}{*}{-0.35}&IT&4&9&9&9&15&9&9&4&4&4&4&4\\
   &CPU&74.9940&134.6261&2.6350&2.6163&4.1425&2.5889&2.5731&1.4688&1.4320&1.4589&1.4341&1.4459\\
   \hline
   \multirow{2}{*}{-0.2}&IT&4&9&10&10&14&10&10&4&4&4&4&4\\
   &CPU&71.3762&132.5856&2.9131&2.8304&3.7992&2.8453&2.8254&1.4619&1.4413&1.4515&1.4455&1.4145\\
   \hline
   \multirow{2}{*}{-0.05}&IT&4&10&8&8&11&8&8&4&4&4&4&4\\
   &CPU&72.0856&144.8090&2.3739&2.3684&3.0708&2.3756&2.3361&1.4179&1.4491&1.4141&1.4436&1.4298\\
   \hline
   \multirow{2}{*}{0}&IT&4&10&8&8&10&8&8&4&4&4&4&4\\
   &CPU&71.9337&144.5986&2.3964&2.3432&2.8074&2.3352&2.3504&1.4293&1.4401&1.4586&1.4487&1.4476\\
   \hline
   \multirow{2}{*}{0.05}&IT&4&10&8&8&11&8&8&4&4&4&4&4\\
   &CPU&72.1193&142.7633&2.3727&2.3463&3.0807&2.3879&2.3307&1.4140&1.4460&1.4291&1.4217&1.4300\\
   \hline
   \multirow{2}{*}{0.2}&IT&4&9&10&10&14&10&10&4&4&4&4&4\\
   &CPU&72.8750&133.7942&2.9122&2.8388&3.8541&2.8803&2.8312&1.4597&1.4336&1.4463&1.4438&1.4127\\
   \hline
   \multirow{2}{*}{0.35}&IT&4&9&9&9&15&9&9&4&4&4&4&4\\
   &CPU&71.9151&133.6850&2.6592&2.5762&4.1201&2.5861&2.5914&1.3834&1.4862&1.4082&1.4407&1.4544\\
   \hline
   \end{tabular}
   }
\end{table}

For Example 7.2, we test the performance of our methods for different $\omega$ and $\eta$. The numerical results are shown in Table 7.2. We can see that Newton's method converges fastest among all proposed methods, but it consumes too much time. Among all doubling algorithms, pADDA converges faster than ADDA and pSDA converges faster than SDA; ADDAn, SDAn and DAn converge faster than ADDA and SDA most of time; pADDAn, pSDAn and pDAn convege faster than ADDA, SDA, ADDAn, SDAn and DAn most of time. DAn chooses the faster method between ADDAn and SDAn, and pDAn chooses the faster method between pADDAn and pSDAn  most of time. We must note that pADDAn may converge slower than pADDA. We observe that pADDA, pSDA, pADDAn, pSDAn and pDAn converge fastest among all doubling algorithms for our example as a whole. Abnormal situations may be explained from the fact that the smaller parameters may not result in faster algorithms for doubling algorithms. From the results, we can observe that the computational work of the bisection procedure is negligible compared to that of the doubling algorithms in the preprocessing procedure.

\section{Conclusions}

In this paper, based on a new parameterized definition of the comparison matrix of a given complex matrix, we propose a new class of complex nonsymmetric algebraic Riccati equations (NAREs) which extends the class of nonsymmetric algebraic Riccati equations proposed by \cite{Axe1}. We also generalize the definition of the extremal solution of the NARE and show that the extremal solution exists and is unique. Besides, we show that Newton's method for solving the NARE is quadratically convergent and the fixed-point iterative methods are linearly convergent. We also give some concrete parameters selection strategies such that the doubling algorithms, including ADDA and SDA, can be used to deliver the extremal solution, and show that the two doubling algorithms with suitable parameters are quadratically convergent. Furthermore, some invariants of the doubling algorithms are also analyzed.

However we fail applying the structure-preserving doubling algorithm with shrink-shift (SDA-ss) to solve the NAREs proposed by us. We will be devoted to choosing suitable parameters such that SDA can be applied to solve the NAREs.

\section*{Acknowledgements}

The work was supported by National Natural Science Foundation of China (11671318).

\end{document}